\documentclass[10pt,oneside]{article}
 \usepackage[a4paper,centering,bindingoffset=0cm,hmargin=3cm,vmargin=3cm,includeheadfoot]{geometry}
\usepackage[font=footnotesize,labelfont=bf,width=0.75\textwidth,labelsep=period]{caption}
\usepackage{amsmath}
\usepackage{amssymb,bbm,color}
\usepackage{graphicx}
\usepackage{authblk}
\usepackage{bm}

\usepackage[hyperref,amsmath,thmmarks]{ntheorem}
\usepackage{aliascnt}

\usepackage[pdftex,colorlinks=true,linkcolor=blue,unicode,pdfpagelabels]{hyperref}

\let\RE\Re
\let\Re=\undefined
\DeclareMathOperator{\Re}{\RE e}
\let\IM\Im
\let\Im=\undefined
\DeclareMathOperator{\Im}{\IM m}
\newcommand{\R}{\mathbbm R}

\newcommand{\SQ}{S}

\newcommand{\Di}{\Omega_{int}}
\newcommand{\De}{\Omega_{ext}}
\newcommand{\Hi}{\bm H^{int}}
\newcommand{\Ei}{\bm E^{int}}
\newcommand{\He}{\bm H^{ext}}
\newcommand{\Ee}{\bm E^{ext}}
\newcommand{\hi}{\bm h^{int}}
\newcommand{\ei}{\bm e^{int}}
\newcommand{\he}{\bm h^{ext}}
\newcommand{\ee}{\bm e^{ext}}
\newcommand{\Hin}{\bm H^{inc}}
\newcommand{\Ein}{\bm E^{inc}}
\newcommand{\hin}{\bm h^{inc}}
\newcommand{\ein}{\bm e^{inc}}
\newcommand{\Hsc}{\bm H^{sc}}
\newcommand{\Esc}{\bm E^{sc}}
\newcommand{\hsc}{\bm h^{sc}}
\newcommand{\esc}{\bm e^{sc}}
\newcommand{\n}{\bm{\hat{n}}}
\newcommand{\di}{\bm{\hat{d}}}
\newcommand{\p}{\bm{\hat{p}}}
\newcommand{\cee}{e_3^{ext}}
\newcommand{\che}{h_3^{ext}}
\newcommand{\cei}{e_3^{int}}
\newcommand{\chii}{h_3^{int}}
\newcommand{\ta}{\bm{\hat{\tau}}}

\let\div=\undefined
\DeclareMathOperator{\div}{\nabla \cdot}
\DeclareMathOperator{\curl}{\nabla \times}

\newaliascnt{proposition}{lemma}

\aliascntresetthe{proposition}

\newaliascnt{theorem}{lemma}
\newtheorem{theorem}[theorem]{Theorem}
\aliascntresetthe{theorem}

\newaliascnt{assumption}{lemma}

\aliascntresetthe{assumption}

\newaliascnt{remark}{lemma}

\aliascntresetthe{remark}

\theoremstyle{nonumberplain}
\theoremseparator{:}
\theoremheaderfont{\normalfont\itshape}
\theorembodyfont{\normalfont}
\theoremsymbol{\ensuremath{\square}}
\newtheorem{proof}{Proof}

\makeatletter
\newcommand{\logmessage}[1]{\@latex@warning{#1}}
\makeatother

\providecommand{\keywords}[1]{\small\textbf{Keywords} #1}

\begin{document}

\title{The direct scattering problem of obliquely incident electromagnetic waves by a penetrable homogeneous cylinder}

\author[1]{Drossos Gintides\thanks{dgindi@math.ntua.gr}}
\author[2]{Leonidas Mindrinos\thanks{leonidas.mindrinos@univie.ac.at}}
\affil[1]{\small Department of Mathematics, National Technical University of Athens, Greece.}
\affil[2]{Computational Science Center, University of Vienna,  Austria.}

\renewcommand\Authands{ and }
\normalsize

\date{}

\maketitle

\begin{abstract}
In this paper we consider the direct scattering problem of obliquely incident time-harmonic electromagnetic plane waves  by an infinitely long dielectric cylinder.  We assume that the cylinder and the outer medium are homogeneous and isotropic. From the symmetry of the problem, Maxwell's equations are reduced to a system of two 2D Helmholtz equations in the cylinder and two Helmholtz equations in the exterior domain coupled on the boundary. We prove uniqueness and existence of this differential system by formulating an equivalent system of integral equations using the direct method.  We transform this system into a Fredholm type system of boundary  integral equations taking advantage of Maue's formula for hypersingular operators. Applying a collocation  method we derive an efficient numerical scheme and provide accurate numerical results using as test cases transmission problems corresponding to  analytic fields derived from  fundamental solutions.

\vspace{0.2cm}
\keywords{direct electromagnetic scattering, oblique incidence, integral equation method, hypersingular operator}
\end{abstract}

    \section{Introduction}
    An interesting area of electromagnetism for its applications and the arising theoretical problems is  the scattering process from obliquely incident time-harmonic plane waves by an infinitely long cylinder. The basic waves in the propagation domain satisfy Maxwell's equations \cite{CakCol06, ColKre83, Mon03, Ned01} and due to the symmetry of the problem it is equivalent to find two scalar fields satisfying  a pair of two-dimensional  Helmholtz equations with different wavenumbers. The complication appears in the boundary conditions. Even for the case of a perfect conductor, in the boundary conditions appear tangential derivatives which make the analysis more difficult.
     There are many studies providing analytical  or numerical solutions \cite{CanLee91, LucPanSche10, Roj88, Tsa07, TsiAliAnaKak07, Wai55, YanGorKis95, YouEls97}.  The proposed methods are based on specific geometries or well known numerical schemes without examining the well-posedness of the corresponding boundary value problem.
     
     Recently, Wang and Nakamura \cite{WanNak12}  used a more elegant theoretical analysis to prove well-posedness of the problem based on the integral equation approach. They proved theoretical and numerical results for the case of homogeneous impedance cylinder using  integral equations. For the theoretical analysis they used properties of the Cauchy singular  integrals and proved that the derived system is of Fredholm type with index zero. For the numerical results they applied a specific decomposition of  the kernels and formulations using  Hilbert's and Symm's integral operators. Considering trigonometric interpolation, they  introduced an efficient numerical scheme.  
     
     The case for general dielectric cylinders is not considered yet, however the same authors, in a later work \cite{NakWan13},  investigated a more complicated model having also a non-homogeneous part, in the sense that the   permittivity and the permeability of the exterior medium are non-constants and smooth in a bounded domain surrounding the cylinder. The main theoretical analysis providing uniqueness and existence in  non-homogeneous   materials is much harder. For the well-posedness they used  the Lax-Phillips method \cite{Isa06}.
     
     In this work, we examine the case of infinite dielectric cylinder illuminated by a transverse magnetic polarized electromagnetic plane wave, known as  oblique incidence. More precisely, in the second Section starting from Maxwell's equations we describe initially the derivation of the  mathematical model for the scattering process from obliquely incident time-harmonic plane waves for the case of infinite inhomogeneous cylinder. We assume that transmission conditions hold on the boundary. The  boundary conditions involve normal and tangential derivatives of the fields.  
     
     In Section 3, we formulate the direct problem in differential form. We derive the Helmholtz equations and the exact form of the boundary conditions in the case of homogeneous cylinder. We prove that the problem is uniquely solvable using Green's formulas and Rellich's lemma. Considering the direct method, initially applied in transmission problems in \cite{CosSte85, KitKle75, KleMar88}, we formulate the problem into an equivalent system of integral equations. We show that this system is of Fredholm type in an appropriate Sobolev space setting. Due to uniqueness of the boundary value problem existence follows from Fredholm alternative. The system consists of compact,  singular and  hypersingular operators. We consider Maue's formula \cite{Mau49}, as in the case of the normal derivative of the double layer potential, to reduce the hypersingularity of the tangential derivative of the double layer potential.

      In the last Section we investigate numerically the problem by a collocation method based on Kress's method for two dimensional integral equation with strongly singular operators \cite{Kre95}. We transform the system of integral equations to a linear system by parametrizing the operators and considering well-known quadrature rules. We derive accurate numerical results for the four fields, interior and exterior, and we compute numerically the far-field patterns of the two exterior fields computed for a specific boundary value problem. Namely, we consider boundary data corresponding to analytic fields derived from point sources, where the interior and exterior fields have singularities outside of their domain of consideration.
     
     \section{Formulation of the direct scattering problem for an inhomogeneous cylinder}
     We consider the scattering problem of an electromagnetic wave by a penetrable cylinder in $\R^3$. Let $\bm x = (x,y,z) \in \R^3 ,$ then we model the cylinder as 
$\Di = \{ \bm x : (x,y) \in \Omega , z \in \R \},$
where $\Omega$ is a bounded domain in $\R^2$ with smooth boundary $\Gamma .$ The cylinder $\Di$ is oriented parallel to the $z$-axis and $\Omega$ is its horizontal cross section. We assume constant permittivity $\epsilon_0$ and permeability $\mu_0$ for the exterior domain $\De : = \R^3 \setminus \overline{\Omega}_{int}.$ The interior domain $\Di$ is characterized by the electric constants $\mu (\bm x )= \mu (x,y)$ and $\epsilon (\bm x)= \epsilon (x,y)$ for all $(x,y) \in \Omega, \, z \in \R .$ 

We define for $\bm x \in \De, \, t\in \R$ the magnetic field $\He (\bm x , t)$ and electric field $\Ee (\bm x , t)$ and equivalently the interior fields  $\Hi (\bm x , t)$ and $\Ei (\bm x , t)$ for $\bm x \in \Di, \, t\in \R .$ Then, these fields satisfy the Maxwell's equations
\begin{equation}\label{eq_Maxwell}
  \begin{aligned}
\curl \Ee + \mu_0 \frac{\partial \He}{\partial t} &= 0, & \curl \He - \epsilon_0 \frac{\partial \Ee}{\partial t} &= 0,  & \bm x \in \De ,\\
\curl \Ei + \mu \frac{\partial \Hi}{\partial t} &= 0, & \curl \Hi - \epsilon \frac{\partial \Ei}{\partial t} &= 0,  
& \bm x \in \Di . 
\end{aligned}
\end{equation}
On the boundary $\Gamma ,$ we consider transmission conditions
\begin{equation*}
\n \times \Ei  =   \n \times \Ee, \quad \n \times \Hi  =   \n \times \He,  \quad \bm x\in \Gamma ,
\end{equation*}
where $\n$ is the outward normal vector, directed into $\De$. 

In order to take advantage of the symmetry of the specific medium, we probe the cylinder with an incident transverse magnetic (TM) polarized electromagnetic plane wave, the so-called oblique incidence in the literature. An arbitrary time-harmonic incident electromagnetic plane wave has the form:
\begin{equation*}
\begin{aligned}
\Ein (\bm x,t ;\di , \p ) &= \frac1{k_0^2 \sqrt{\epsilon_0}} \curl \curl \left( \p \, e^{ik_0 \bm x \cdot  \di}\right) e^{-i\omega t}, \\
\Hin (\bm x,t ;\di , \p ) &= \frac1{ik_0 \sqrt{\mu_0}} \curl \left( \p \, e^{ik_0 \bm x \cdot  \di}\right) e^{-i\omega t},
\end{aligned} 
\end{equation*}
where $\omega >0$ is the frequency, $k_0 = \omega \sqrt{\mu_0 \epsilon_0}$ is the wave number in the exterior domain, $\p$ is the polarization vector and $\di$ the vector describing the incident direction, satisfying $\di \perp \p .$ 

In the following, due to the linearity of the problem we suppress the time-dependence and we consider the fields only as functions of the space variable $\bm x$. In order to describe the incident fields for the specific TM polarization, we define by $\theta$  the incident angle with respect to the negative $z$ axis and by $\phi$ the polar angle of $\di$ (in spherical coordinates), then, $\di = (\sin \theta \cos \phi , \sin \theta \sin \phi , -\cos \theta )$ and $\p = (\cos \theta \cos \phi , \cos \theta \sin \phi , \sin \theta ), $ assuming that $\theta \in (0,\pi /2) \cup (\pi/2 , \pi) .$ Hence, we obtain
\begin{equation*}
\begin{aligned}
\Ein (\bm x ;\di , \p ) &= \frac1{\sqrt{\epsilon_0}} \, \di \times \p \times \di \, e^{ik_0 \bm x \cdot  \di} = \frac1{\sqrt{\epsilon_0}} \, \p  \, e^{ik_0 \bm x \cdot  \di}, \\
\Hin (\bm x  ;\di , \p ) &= \frac1{\sqrt{\mu_0}} \di \times \p \, e^{ik_0 \bm x \cdot  \di} = \frac1{\sqrt{\mu_0}} (\sin \phi , -\cos \phi ,0) \, e^{ik_0 \bm x \cdot  \di}.
\end{aligned} 
\end{equation*}

Taking into account the cylindrical symmetry of the medium and the $z$-independence of the electric coefficients we express the incident fields as separable functions of $(x,y)$ and $z.$ Thus, we define $\beta = k_0 \cos \theta$ and $\kappa_0 = \sqrt{k_0^2 - \beta^2} = k_0 \sin \theta$ and it follows that the incident fields can be decomposed to
\begin{equation}\label{eq_incident3}
\begin{aligned}
\Ein (\bm x ;\di , \p ) = \ein (x,y) \, e^{-i\beta z}, \quad
\Hin (\bm x  ;\di , \p ) = \hin (x,y) \, e^{-i\beta z},
\end{aligned} 
\end{equation}
where
\begin{align*}
\ein (x,y) &= \frac1{\sqrt{\epsilon_0}} \, \p  \, e^{i\kappa_0 (x \cos \phi + y \sin \phi )}, \\
\hin (x,y) &= \frac1{\sqrt{\mu_0}} \, (\sin \phi , -\cos \phi ,0) \, e^{i\kappa_0 (x \cos \phi + y \sin \phi )}.
\end{align*}

Now, we are in position to transform equations \eqref{eq_Maxwell} into a system of equations only for the $z$-component of the electric and magnetic fields. Firstly, we see that for the specific illumination of the form \eqref{eq_incident3}, using separation of variables, also the scattered fields take the form:
\begin{equation*}
\begin{aligned}
\Esc (\bm x ;\di , \p ) = \esc (x,y) \, e^{-i\beta z}, \quad 
\Hsc (\bm x  ;\di , \p ) = \hsc (x,y) \, e^{-i\beta z}, \quad \bm x \in \De,
\end{aligned} 
\end{equation*}
where $\esc = ( e_1^{sc}, e_2^{sc},e_3^{sc} )$ and $\hsc = ( h_1^{sc} , h_2^{sc} ,h_3^{sc} ) .$ Then, the exterior fields are given by
\begin{equation*}
\begin{aligned}
\Ee (\bm x ;\di , \p ) &= \left( \esc (x,y) + \ein (x,y) \right) \, e^{-i\beta z} = \ee (x,y) \, e^{-i\beta z}, & \bm x \in \De ,\\
\He (\bm x  ;\di , \p ) &= \left( \hsc (x,y) + \hin (x,y) \right) \, e^{-i\beta z} = \ee (x,y) \, e^{-i\beta z}, & \bm x \in \De .
\end{aligned} 
\end{equation*}
Equivalently, the interior fields are represented by
\begin{equation*}
\begin{aligned}
\Ei (\bm x ;\di , \p ) = \ei (x,y) \, e^{-i\beta z}, \quad 
\Hi (\bm x  ;\di , \p ) = \hi (x,y) \, e^{-i\beta z}, \,\, \bm x \in \Di ,
\end{aligned} 
\end{equation*}
where $\ei  = ( e_1^{int} , e_2^{int}  , e_3^{int} )$ and $\hi  = ( h_1^{int} , h_2^{int} , h_3^{int} ) .$

For any field of the form
\begin{equation*}
\begin{aligned}
\bm E (\bm x ;\di , \p ) = \bm e (x,y) \, e^{-i\beta z}, \quad
\bm H (\bm x  ;\di , \p ) = \bm h (x,y) \, e^{-i\beta z}, \quad \bm x \in \R^3 ,
\end{aligned} 
\end{equation*}
we consider the Maxwell's equations in $\R^3$ for arbitrary $\epsilon , \mu$ and $k^2 =  \mu \epsilon \omega^2 - \beta^2$ (remark here the space dependence of $\epsilon, \mu ).$ Then, following \cite{NakWan13} we obtain the relations
\begin{equation}\label{eq_variables}
\begin{aligned}
e_1  (x,y) = -\frac1{k^2} \left( i\beta \frac{\partial e_3}{\partial x} (x,y) - i\mu \omega \frac{\partial h_3}{\partial y} (x,y)\right), \\
e_2  (x,y) = -\frac1{k^2} \left( i\beta \frac{\partial e_3}{\partial y} (x,y) + i\mu \omega \frac{\partial h_3}{\partial x} (x,y)\right), \\
h_1  (x,y) = -\frac1{k^2} \left( i\beta \frac{\partial h_3}{\partial x} (x,y) + i\epsilon \omega \frac{\partial e_3}{\partial y} (x,y)\right), \\
h_2  (x,y) = -\frac1{k^2} \left( i\beta \frac{\partial h_3}{\partial y} (x,y) - i\epsilon \omega \frac{\partial e_3}{\partial x} (x,y)\right). \\
\end{aligned}
\end{equation}

 Substituting \eqref{eq_variables} in \eqref{eq_Maxwell}, we have that the pair $(e_3 , h_3)$ satisfies the equations
\begin{equation*}
\begin{aligned}
\frac{k^2}{\epsilon \omega } \div \left( \frac{\epsilon \omega}{k^2} \nabla e_3\right) + \frac{k^2}{\epsilon \omega } J  \, \nabla \left( \frac{\beta}{k^2}\right) \cdot \nabla h_3 + k^2 e_3 &= 0, \\
  \frac{k^2}{\mu \omega } \div \left( \frac{\mu \omega}{k^2} \nabla h_3\right) - \frac{k^2}{\mu \omega } J  \, \nabla \left( \frac{\beta}{k^2}\right) \cdot \nabla e_3 + k^2 h_3 &= 0,
\end{aligned}
\end{equation*}
where
\[
\bm J = \begin{pmatrix}
\phantom{-}0 & 1\\ -1 & 0
\end{pmatrix} .
\]

The interior and the exterior domains are characterized by different wavenumbers, given by
\begin{equation*}
k^2 (\bm x)= \left\{
     \begin{array}{lr}
     k_{int}^2 (\bm x) := \mu (x,y)\, \epsilon (x,y)\, \omega^2 - \beta^2 ,  & \bm x \in \Di ,\\
        k_{ext}^2 (\bm x) := \mu_0 \epsilon_0\omega^2 - \beta^2 = \kappa^2_0,  & \bm x \in \De .
     \end{array}
   \right.
\end{equation*}

In this section, for completeness in the formulation of the direct problem we keep the space dependence of $k_{int}.$ Later, we consider only the case of constant parameters. Here, we have to assume that $\mu (\bm x)\epsilon (\bm x) > \epsilon_0 \mu_0 \cos \theta$
in order to have $\inf_{\bm x} k_{int}^2 (\bm x) >0.$ Thus, the fields $\cee (x,y)$ and $\che (x,y)$ satisfy
\begin{equation}\label{maxwell_ext}
\begin{aligned}
\Delta \cee + \kappa^2_0 \,\cee = 0, \quad 
\Delta \che + \kappa^2_0 \,\che = 0, \quad \bm x \in \De ,
\end{aligned}
\end{equation}
and the interior fields
\begin{equation}\label{maxwell_int}
\begin{aligned}
\frac{k_{int}^2 (\bm x)}{\epsilon (\bm x)} \div \left( \frac{\epsilon (\bm x)}{k_{int}^2 (\bm x) } \nabla \cei \right) + \frac{k_{int}^2 (\bm x)}{\epsilon (\bm x) \omega } \bm J  \, \nabla \left( \frac{\beta}{k_{int}^2 (\bm x)}\right) \cdot \nabla \chii    
+ k_{int}^2 (\bm x) \, \cei = 0, \quad \bm x &\in \Di , \\
 \frac{k_{int}^2 (\bm x)}{\mu (\bm x)} \div \left( \frac{\mu (\bm x)}{k_{int}^2 (\bm x) } \nabla \chii \right) - \frac{k_{int}^2 (\bm x)}{\mu (\bm x) \omega } \bm J  \, \nabla \left( \frac{\beta}{k_{int}^2 (\bm x)}\right) \cdot \nabla \cei 
 + k_{int}^2 (\bm x) \, \chii = 0, \quad \bm x &\in \Di .
\end{aligned}
\end{equation}

Now, we are going to derive the exact form of the boundary conditions. We introduce the notations: $\bm e_t  = \bm{\hat{x}} \, e_1 + \bm{\hat{y}} \, e_2 , \, \bm h_t  = \bm{\hat{x}} \, h_1 + \bm{\hat{y}} \, h_2 ,$ and $\nabla_t = \bm{\hat{x}} \tfrac{\partial}{\partial x}+\bm{\hat{y}} \frac{\partial}{\partial y},$ where $\bm{\hat{x}}, \bm{\hat{y}}, \bm{\hat{z}}$ denote the unit vectors in $\R^2 .$ Let $(\n , \ta , \bm{\hat{z}})$ be a local coordinate system, where $\n = (n_1 , n_2)$ is the outward normal vector and $\ta = (-n_2 , n_1)$ the outward tangent vector on $\Gamma .$ Then, from \eqref{eq_variables} we obtain
\begin{equation}\label{eq_tangential}
\begin{aligned}
\ta \cdot \bm e_t &= -\frac1{k^2} \left( i\mu \omega \n \cdot \nabla_t h_3 + i\beta \ta \cdot \nabla_t e_3 \right), \\
\ta \cdot \bm h_t &= -\frac1{k^2} \left( -i\epsilon \omega \n \cdot \nabla_t e_3 + i\beta \ta \cdot \nabla_t h_3 \right), \\ 
\end{aligned}
\end{equation}
using that $\ta \cdot \left( \bm{\hat{z}} \times \nabla_t \cdot \right) = \n \cdot \nabla_t \cdot .$ 

We observe, setting zero to the $z$-component of $\n , \ta$ in $\R^3 ,$ that
\begin{equation*}
\begin{aligned}
\n \times \bm E = - e_3 \ta + (n_1 e_2 - n_2 e_1 ) \, \bm{\hat{z}}, \quad
\n \times \bm H = - h_3 \ta + (n_1 h_2 - n_2 h_1 ) \, \bm{\hat{z}}.
\end{aligned}
\end{equation*}
Then from \eqref{eq_variables} and \eqref{eq_tangential}, we derive
\begin{equation*}
\begin{aligned}
\n \times \Ee = - \cee \ta + \ta \cdot \ee_t \, \bm{\hat{z}}, \quad
\n \times \He = - \che \ta + \ta \cdot \he_t \, \bm{\hat{z}},
\end{aligned}
\end{equation*}
for the exterior fields, where $\ee_t := \bm{\hat{x}} \, e^{ext}_1 + \bm{\hat{y}} \, e^{ext}_2 , \,\he_t := \bm{\hat{x}} \, h^{ext}_1 + \bm{\hat{y}} \, h^{ext}_2 ,$  and
\begin{equation*}
\begin{aligned}
\n \times \Ei = - \cei \ta + \ta \cdot \ei_t \, \bm{\hat{z}},\quad
\n \times \Hi = - \chii \ta + \ta \cdot \hi_t \, \bm{\hat{z}}.
\end{aligned}
\end{equation*}
for the interior fields, where $\ei_t := \bm{\hat{x}} \, e^{int}_1 + \bm{\hat{y}} \, e^{int}_2 , \,\hi_t := \bm{\hat{x}} \, h^{int}_1 + \bm{\hat{y}} \, h^{int}_2$. 

Here, we observe that the tangential forms of the fields can be written in terms of $\ta$ and $\bm{\hat{z}},$ two linear independent vectors. Thus, the boundary condition
\[
\n \times \Ei  =   \n \times \Ee , \quad \bm x\in \Gamma ,
\]
is equivalent to the system
\begin{equation*}
\begin{aligned}
\cei = \cee ,   \quad
\ta \cdot \ei_t = \ta \cdot \ee_t, \quad \bm x\in \Gamma ,
\end{aligned}
\end{equation*}
and equivalently for the magnetic fields
\begin{equation*}
\begin{aligned}
\chii = \che , \quad
\ta \cdot \hi_t = \ta \cdot \he_t, \quad \bm x\in \Gamma .
\end{aligned}
\end{equation*}
We define
\[
\frac{\partial}{\partial n } = \n \cdot \nabla_t , \quad \frac{\partial}{\partial \tau } = \ta \cdot \nabla_t
\]
and we rewrite the above boundary conditions as
\begin{equation}\label{boundary_electric}
\begin{aligned}
\cei &= \cee , & \bm x\in \Gamma , \\
\frac{\mu (\bm x)}{k_{int}^2 (\bm x)} \omega \frac{\partial \chii}{\partial n }  + \frac{\beta}{k_{int}^2 (\bm x)} \frac{\partial \cei}{\partial \tau } &= \frac{\mu_0}{\kappa^2_0} \omega \frac{\partial \che}{\partial n }  + \frac{\beta}{\kappa^2_0} \frac{\partial \cee}{\partial \tau }, & \bm x\in \Gamma ,
\end{aligned}
\end{equation}
and
\begin{equation}\label{boundary_magnetic}
\begin{aligned}
\chii &= \che , & \bm x\in \Gamma , \\
\frac{\epsilon (\bm x)}{k_{int}^2 (\bm x)} \omega \frac{\partial \cei}{\partial n }  - \frac{\beta}{k_{int}^2 (\bm x)} \frac{\partial \chii}{\partial \tau } &= \frac{\epsilon_0}{\kappa^2_0} \omega \frac{\partial \cee}{\partial n }  - \frac{\beta}{\kappa^2_0} \frac{\partial \che}{\partial \tau }, & \bm x\in \Gamma .
\end{aligned}
\end{equation}

To ensure that the scattered fields are outgoing, the components must satisfy in addition the radiation conditions in $\R^2:$
\begin{equation}\label{radiation}
\begin{aligned}
\lim_{r \rightarrow \infty} \sqrt{r} \left( \frac{\partial e_3^{sc}}{\partial r} - i\kappa_0 e_3^{sc} \right) =0 , \quad
 \lim_{r \rightarrow \infty} \sqrt{r} \left( \frac{\partial h_3^{sc}}{\partial r} - i\kappa_0 h_3^{sc} \right) =0 , \\
\end{aligned}
\end{equation}   
where $r = |(x,y)|,$ uniformly over all directions.

Thus, the direct transmission problem for oblique incident wave, is to find the fields $\chii , h_3^{sc}, \cei$ and $e_3^{sc}$ which satisfy equations \eqref{maxwell_ext} and \eqref{maxwell_int}, the transmission conditions \eqref{boundary_electric} and \eqref{boundary_magnetic} and the radiation conditions \eqref{radiation}.

We remark here that since we consider TM polarized wave, see equation \eqref{eq_incident3}, the incident fields for $\bm x \in \De$ are simplified to
\begin{equation}\label{incident_el}
\begin{aligned}
e_3^{inc} (x,y) = \frac1{\sqrt{\epsilon_0}} \sin \theta \, e^{i\kappa_0 (x \cos \phi +y \sin \phi )}, \quad
h_3^{inc} (x,y) = 0.
\end{aligned}
\end{equation}
     
      \section{The direct problem for a homogeneous cylinder using the integral equation method}
      From now on, $\bm x \in\R^2 .$ In this section we consider the simplified version where $\mu (\bm x) = \mu_1$ and $\epsilon (\bm x) = \epsilon_1$ are constant in the interior domain. To simplify the following analysis, we set $\Omega_1 = \Omega \subset \R^2, \, \Omega_0 = \R^2 \setminus \Omega$ and 
\begin{equation*}
\begin{aligned}
u_0 (\bm x) &= e_3^{sc} (\bm x), & v_0 (\bm x) &= h_3^{sc} (\bm x), &\bm x \in \Omega_0 ,\\
 u_1 (\bm x) &= \cei (\bm x),  & v_1 (\bm x) &= \chii (\bm x), &\bm x \in \Omega_1 .
\end{aligned}
\end{equation*}
In the following, $j = 0,1$ counts for the exterior ($\bm x\in \Omega_0$) and interior domain ($\bm x\in \Omega_1$), respectively. Then, the direct scattering problem, presented in the previous section, is modified to
\begin{equation}\label{maxwell_constant}
\begin{aligned}
\Delta u_j + \kappa^2_j \, u_j &= 0, & \Delta v_j + \kappa^2_j \, v_j &= 0, &\bm x \in \Omega_j,
\end{aligned}
\end{equation}
for $j=0,1$ where $\kappa^2_1 =  \mu_1 \epsilon_1 \omega^2 - \beta^2 ,$ with boundary conditions
\begin{subequations}\label{boundary_constant}
\begin{align}
u_1 &= u_0 + e_3^{inc}, & \bm x\in \Gamma , \label{eq_diri1}\\
\tilde\mu_1 \omega \frac{\partial v_1}{\partial n }  + \beta_1 \frac{\partial u_1}{\partial \tau } &= \tilde\mu_0 \omega \frac{\partial v_0}{\partial n }  + \beta_0 \frac{\partial u_0}{\partial \tau }+ \beta_0 \frac{\partial e_3^{inc}}{\partial \tau }, & \bm x\in \Gamma, \label{26b}\\ 
v_1 &= v_0 , & \bm x\in \Gamma , \label{eq_diri2}\\
\tilde\epsilon_1 \omega \frac{\partial u_1 }{\partial n }  - \beta_1 \frac{\partial v_1}{\partial \tau } &= \tilde\epsilon_0 \omega \frac{\partial u_0}{\partial n } +\tilde\epsilon_0 \omega \frac{\partial e_3^{inc}}{\partial n }  - \beta_0 \frac{\partial v_0}{\partial \tau }, & \bm x\in \Gamma \label{26d},
\end{align}
\end{subequations}
where $\tilde\mu_j = \mu_j / \kappa_j^2 , \, \tilde\epsilon_j = \epsilon_j / \kappa_j^2 , \, \beta_j = \beta / \kappa_j^2$   and the radiation conditions
\begin{equation}\label{radiation_constant}
\begin{aligned}
\lim_{r \rightarrow \infty} \sqrt{r} \left( \frac{\partial u_0}{\partial r} - i\kappa_0 u_0 \right) =0 , \quad
 \lim_{r \rightarrow \infty} \sqrt{r} \left( \frac{\partial v_0}{\partial r} - i\kappa_0 v_0 \right) =0 . \\
\end{aligned}
\end{equation}

\begin{theorem}\label{theorem_most}
If $\kappa_1^2$ is not an interior Dirichlet eigenvalue, then the problem \eqref{maxwell_constant} - \eqref{radiation_constant} has at most one solution.
\end{theorem}

\begin{proof}
It is sufficient to show that if $u_0 , v_0 ,u_1 ,v_1$ solve the homogeneous problem \eqref{maxwell_constant} - \eqref{radiation_constant}, that is for $e_3^{inc}=0,$ then, $u_0 = v_0 = 0$ in $\Omega_0$ and $u_1 = v_1 = 0$ in $\Omega_1 .$ Let $S_r$ be a disk with radius $r ,$ boundary $\Gamma_r ,$ centered at the origin and containing $\Omega_1 .$ We set $\Omega_r = S_r \setminus \overline{\Omega}_1 ,$ see Figure \ref{Fig1}. 

\begin{figure}
\begin{center}
\includegraphics[scale=1]{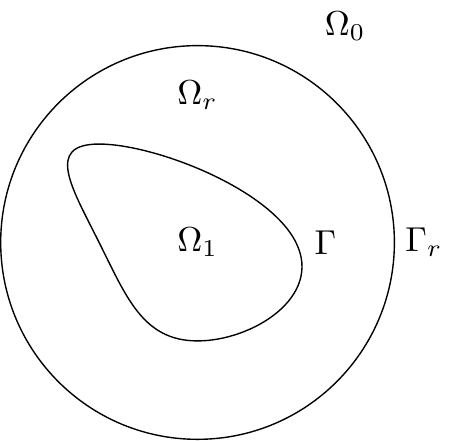}
\caption{The set $\Omega_r$.}\label{Fig1}
\end{center}
\end{figure}

The boundary conditions of the homogeneous problem read
\begin{equation}\label{boundary_homo}
\begin{aligned}
u_1 &= u_0 , & \bm x\in \Gamma , \\
\tilde\mu_1  \frac{\partial v_1}{\partial n } - \tilde\mu_0  \frac{\partial v_0}{\partial n }  &= - \frac{\beta_1}\omega  \frac{\partial u_1}{\partial \tau }   + \frac{\beta_0}\omega \frac{\partial u_0}{\partial \tau } , & \bm x\in \Gamma, \\ 
v_1 &= v_0 , & \bm x\in \Gamma , \\
\tilde\epsilon_1  \frac{\partial u_1 }{\partial n } - \tilde\epsilon_0  \frac{\partial u_0}{\partial n } &=  \frac{\beta_1}\omega \frac{\partial v_1}{\partial \tau }   - \frac{\beta_0}\omega \frac{\partial v_0}{\partial \tau }, & \bm x\in \Gamma .
\end{aligned}
\end{equation} 
We apply Green's first identity in $\Omega_1 $ and considering \eqref{maxwell_constant} we obtain 
\begin{equation}\label{green_omega1}
\begin{aligned}
\tilde\epsilon_1 \int_\Gamma u_1 \frac{\partial \overline{u}_1 }{\partial n } \, ds  &= \tilde\epsilon_1 \int_{\Omega_1}  \left( |\nabla u_1  |^2 + u_1 \Delta \overline{u}_1 \right) d \bm x \\
&= \tilde\epsilon_1 \int_{\Omega_1}  \left( |\nabla u_1  |^2 -\kappa_1^2 | u_1 |^2 \right) d \bm x ,\\
\tilde\mu_1 \int_\Gamma v_1 \frac{\partial \overline{v}_1 }{\partial n } \, ds  &= \tilde\mu_1 \int_{\Omega_1}  \left( |\nabla v_1  |^2 + v_1 \Delta \overline{v}_1 \right) d \bm x \\
&= \tilde\mu_1 \int_{\Omega_1}  \left( |\nabla v_1  |^2 -\kappa_1^2 | v_1 |^2 \right) d \bm x .
\end{aligned}
\end{equation}

Similarly, Green's first identity in $\Omega_r$ together with equations \eqref{boundary_homo} and \eqref{green_omega1} gives
\begin{align*}
\tilde\epsilon_0 \int_{\Gamma_r} u_0 \frac{\partial \overline{u}_0 }{\partial n } \, ds  &= \tilde\epsilon_0 \int_{\Omega_r}  \left( |\nabla u_0  |^2 + u_0 \Delta \overline{u}_0 \right) d \bm x + \tilde\epsilon_0 \int_\Gamma u_0 \frac{\partial \overline{u}_0 }{\partial n } \, ds \\ 
&= \tilde\epsilon_0 \int_{\Omega_r}  \left( |\nabla u_0  |^2 -\kappa_0^2 | u_0 |^2 \right) d \bm x \\
&+  \int_\Gamma u_0 \left( \tilde\epsilon_1  \frac{\partial \overline{u}_1 }{\partial n } -  \frac{\beta_1}\omega \frac{\partial \overline{v}_1}{\partial \tau }   + \frac{\beta_0}\omega \frac{\partial \overline{v}_0}{\partial \tau } \right)  ds \\
&= \tilde\epsilon_0 \int_{\Omega_r}  \left( |\nabla u_0  |^2 -\kappa_0^2 | u_0 |^2 \right) d \bm x \\
&+ \tilde\epsilon_1 \int_{\Omega_1}  \left( |\nabla u_1  |^2 -\kappa_1^2 | u_1 |^2 \right) d \bm x - \frac{\beta_1}\omega \int_\Gamma u_1   \frac{\partial \overline{v}_1}{\partial \tau }  \, ds +  \frac{\beta_0}\omega  \int_\Gamma u_0 \frac{\partial \overline{v}_0}{\partial \tau }   ds
\end{align*}
and
\begin{align*}
\tilde\mu_0 \int_{\Gamma_r} v_0 \frac{\partial \overline{v}_0 }{\partial n } \, ds  &= \tilde\mu_0 \int_{\Omega_r}  \left( |\nabla v_0  |^2 + v_0 \Delta \overline{v}_0 \right) d \bm x + \tilde\mu_0 \int_\Gamma  v_0 \frac{\partial \overline{v}_0 }{\partial n } \, ds \\
&= \tilde\mu_0 \int_{\Omega_r}  \left( |\nabla v_0  |^2 -\kappa_0^2 | v_0 |^2 \right) d \bm x \\
&+  \int_\Gamma v_0 \left( \tilde\mu_1  \frac{\partial \overline{v}_1 }{\partial n } -  \frac{\beta_0}\omega \frac{\partial \overline{u}_0}{\partial \tau }   + \frac{\beta_1}\omega \frac{\partial \overline{u}_1}{\partial \tau } \right)  ds \\
&= \tilde\mu_0 \int_{\Omega_r}  \left( |\nabla v_0  |^2 -\kappa_0^2 | v_0 |^2 \right) d \bm x \\
&+ \tilde\mu_1 \int_{\Omega_0}  \left( |\nabla v_1  |^2 -\kappa_1^2 | v_1 |^2 \right) d \bm x - \frac{\beta_0}\omega \int_\Gamma v_0   \frac{\partial \overline{u}_0}{\partial \tau }  \, ds +  \frac{\beta_1}\omega  \int_\Gamma v_1 \frac{\partial \overline{u}_1}{\partial \tau } ds .
\end{align*}

We add the above two equations and noting that
\[
-  \int_\Gamma u_1   \frac{\partial \overline{v}_1}{\partial \tau }  \, ds = \overline{ \int_\Gamma v_1 \frac{\partial \overline{u}_1}{\partial \tau }     ds }, \qquad  \int_\Gamma u_0   \frac{\partial \overline{v}_0}{\partial \tau }  \, ds = -\overline{ \int_\Gamma v_0 \frac{\partial \overline{u}_0}{\partial \tau }     ds },
\]
we obtain
\begin{equation*}
\Im \left( \tilde\epsilon_0 \int_{\Gamma_r} u_0 \frac{\partial \overline{u}_0 }{\partial n } \, ds + \tilde\mu_0 \int_{\Gamma_r} v_0 \frac{\partial \overline{v}_0 }{\partial n } \, ds \right) = 0 ,
\end{equation*}
or equivalently, using the radiation conditions (see equation 2.12 \cite{WanNak12})
\[
\lim\limits_{r\rightarrow\infty} \int_{\Gamma_r} \left( \epsilon_0 |u_0|^2 + \tilde\epsilon_0 \left| \frac{\partial u_0 }{\partial n } \right|^2 +\mu_0  |v_0|^2 + \tilde\mu_0 \left| \frac{\partial v_0 }{\partial n } \right|^2 \right)  ds = 0.
\]
Thus
\begin{equation*}
\lim\limits_{r\rightarrow\infty} \int_{\Gamma_r}  |u_0|^2 ds = \lim\limits_{r\rightarrow\infty} \int_{\Gamma_r}   |v_0|^2 ds = 0,
\end{equation*}
and by Rellich's lemma it follows that $u_0 = v_0 = 0$ in $\Omega_0.$ Hence, $u_0 = v_0 = 0$ in $\Gamma$ and $u_1 = v_1 =0$ in $\Gamma$ from the boundary conditions. Then, 
$u_1 = v_1 =0$ in $\Omega_1$ follows from the unique solvability of the interior Dirichlet problem, given the assumption of the theorem.
\end{proof}

We define the fundamental solution of the Helmholtz equation in $\R^2:$
\begin{equation}
\Phi_j (\bm x,\bm y) = \frac{i}4 H_0^{(1)} (\kappa_j |\bm x-\bm y|), \quad \bm x,\bm y \in\Omega_j , \quad \bm x \neq \bm y,  
\end{equation}
where $H_0^{(1)}$ is the Hankel function of the first kind and zero order. For a continuous density $f,$ we introduce the single- and double-layer potentials defined by
\begin{equation}\label{single_double}
\begin{aligned}
(S_j f) (\bm x) &= \int_\Gamma \Phi_j (\bm x,\bm y) f(\bm y) ds (\bm y), & \bm x \in\Omega_j , \\
(D_j f) (\bm x) &= \int_\Gamma \frac{\partial \Phi_j}{\partial n (\bm y)} (\bm x,\bm y) f(\bm y) ds (\bm y), & \bm x \in\Omega_j ,
\end{aligned}
\end{equation}
and their derivatives, normal and tangential as $\bm x \rightarrow \Gamma$, using the standard jump relations, see for example \cite{ColKre83, CouHil62}
\begin{equation}\label{jump_relations}
\begin{aligned}
\frac{\partial }{\partial n } (S_j f) (\bm x) &= \int_\Gamma \frac{\partial \Phi_j}{\partial n (\bm x)} (\bm x,\bm y) f(\bm y) ds (\bm y) \mp \frac12 f(\bm x) 
:= (NS_j f) (\bm x) \mp \frac12 f(\bm x)  , & \bm x \in\Gamma, \\
\frac{\partial }{\partial n } (D_j f) (\bm x) &= \int_\Gamma \frac{\partial^2 \Phi_j}{\partial n (\bm x) \partial n (\bm y)} (\bm x,\bm y) f(\bm y) ds (\bm y) 
:= (ND_j f) (\bm x)   , & \bm x \in\Gamma, \\
\frac{\partial }{\partial \tau } (S_j f) (\bm x) &= \int_\Gamma \frac{\partial \Phi_j}{\partial \tau (\bm x)} (\bm x,\bm y) f(\bm y) ds (\bm y) 
:= (TS_j f) (\bm x)  , & \bm x \in\Gamma, \\
\frac{\partial }{\partial \tau } (D_j f) (\bm x) &= \int_\Gamma \frac{\partial^2 \Phi_j}{\partial \tau (\bm x) \partial n (\bm y)} (\bm x,\bm y) f(\bm y) ds (\bm y) \pm \frac12 \frac{\partial f}{\partial \tau } (\bm x) 
:= (TD_j f) (\bm x) \pm \frac12\frac{\partial f}{\partial \tau } (\bm x) , & \bm x \in\Gamma ,
\end{aligned}
\end{equation}
where the upper (lower) sign indicates the limits obtained by approaching the boundary $\Gamma$ from $\Omega_0 \, (\Omega_1) ,$ this means when $j=0 \, (j=1).$ For the last equation, we assume $f$ to be continuously differentiable. We presented the jump relations for continuous densities only for simplicity, since in the following proof we use a Sobolev space setting.

Here, we have to mention the continuity of the single-layer potentials in $\R^2$ and the discontinuity of the double-layer potentials  $(\pm \tfrac12 f)$.  All the integrals are well defined   and particularly the potentials $S_j , D_j$ and $NS_j $ have weakly singular kernels (logarithmic singularity), the potentials $TS_j$ have Cauchy type singularity (of order $1/|\bm x - \bm y|$) and the potentials $ND_j , TD_j$ are hypersingular (of order $1/|\bm x - \bm y|^2$).

\begin{theorem}\label{theo32}
If $\kappa_1^2$ is not an interior Dirichlet eigenvalue and $\kappa_0^2$ is not an interior Dirichlet and Neumann eigenvalue, then the problem \eqref{maxwell_constant} - \eqref{radiation_constant} has a unique solution.
\end{theorem}

\begin{proof}
We apply the direct method, see for instance \cite{KleMar88}, to transform the problem into a system of integral equations. We consider the Green's second theorem in the interior domain
\begin{align}\label{eq_green0}
- u_1 (\bm x) &= \int_\Gamma \frac{\partial \Phi_1}{\partial n (\bm y)} (\bm x,\bm y) u_1 (\bm y) ds (\bm y) - \int_\Gamma \Phi_1 (\bm x,\bm y) \frac{\partial u_1}{\partial n (\bm y)} (\bm y) ds (\bm y), \\
&= (D_1 u_1 ) (\bm x) - ( S_1 \partial_\eta u_1) (\bm x), \quad \bm x\in \Omega_1 , \nonumber
\end{align}
similarly
\[
- v_1 (\bm x) = (D_1 v_1 ) (\bm x) - ( S_1 \partial_\eta v_1) (\bm x), \quad \bm x\in \Omega_1 ,
\]
and in the exterior domain
\begin{equation}\label{eq_greenD0}
\begin{aligned}
 u_0 (\bm x) &= (D_0 u_0 ) (\bm x) - ( S_0 \partial_\eta u_0 ) (\bm x), & \bm x\in \Omega_0 , \\
  v_0 (\bm x) &= (D_0 v_0 ) (\bm x) - ( S_0 \partial_\eta v_0) (\bm x), & \bm x\in \Omega_0 .
\end{aligned}
\end{equation}
Letting $\bm x \rightarrow \Gamma ,$ in the above formulas and taking the normal and the tangential derivatives on $\Gamma,$ we obtain
\begin{subequations}\label{eq_jumpdound}
\begin{align}
\left( NS_j \pm \tfrac12 I\right) \partial_\eta u_j &= ND_j  u_j , & \left( NS_j \pm \tfrac12 I\right) \partial_\eta v_j &= ND_j  v_j ,\label{36e} \\
\left( D_j \mp \tfrac12 I\right) u_j &= S_j \partial_\eta u_j , & \left( D_j \mp\tfrac12 I\right) v_j &= S_j \partial_\eta v_j ,\label{31d}\\ 
 TD_j u_j -  TS_j \partial_\eta u_j &= \pm \tfrac12 \partial_\tau u_j , &  TD_j v_j -  TS_j \partial_\eta v_j &= \pm \tfrac12 \partial_\tau v_j .
\end{align}
\end{subequations}
Combining the relations in \eqref{31d} for $j=0$ with the boundary conditions \eqref{26d} and \eqref{26b}  respectively, we have
\begin{equation}\label{eq_bound}
\begin{aligned}
\left( D_0 -\frac12 I\right) u_0 &= \frac{\tilde\epsilon_1}{\tilde\epsilon_0} S_0 \partial_\eta u_1 - \frac{\beta_1}{\tilde\epsilon_0 \omega} S_0 \partial_\tau v_1 + \frac{\beta_0}{\tilde\epsilon_0 \omega} S_0 \partial_\tau v_0  - S_0 \partial_\eta e^{inc}_3 ,\\
\left( D_0 -\frac12 I\right) v_0 &= \frac{\tilde\mu_1}{\tilde\mu_0} S_0 \partial_\eta v_1 + \frac{\beta_1}{\tilde\mu_0 \omega} S_0 \partial_\tau u_1 - \frac{\beta_0}{\tilde\mu_0 \omega} S_0 \partial_\tau u_0  -  \frac{\beta_0}{\tilde\mu_0 \omega} S_0 \partial_\tau e^{inc}_3 .
\end{aligned}
\end{equation}
We define
\begin{equation}\label{operqtorK0}
K_j := \left( NS_j \pm \tfrac12 I\right)^{-1} ND_j , \quad L_j : = 2 (TD_j -  TS_j K_j ) .
\end{equation}
The operator $K_0$ is well-defined since the integral equation \eqref{36e} for $j=0$ corresponds to the solution of the interior Neumann problem and we assumed $\kappa_0^2$ not to be interior eigenvalue. Similarly, $K_1$ is well-defined if $\kappa_1^2$ is not an interior Dirichlet eigenvalue.

Then, the system of equations \eqref{eq_bound}, using \eqref{eq_jumpdound}, is transformed to
\begin{equation*}
\begin{aligned}
\left( D_0 -\frac12 I\right) u_0 - \frac{\tilde\epsilon_1}{\tilde\epsilon_0} S_0 K_1 u_1 - \frac{\beta_1}{\tilde\epsilon_0 \omega} S_0 L_1 v_1 - \frac{\beta_0}{\tilde\epsilon_0 \omega} S_0 L_0 v_0 
&= - S_0 \partial_\eta e^{inc}_3 ,\\
\left( D_0 -\frac12 I\right) v_0 - \frac{\tilde\mu_1}{\tilde\mu_0} S_0 K_1 v_1 + \frac{\beta_1}{\tilde\mu_0 \omega} S_0 L_1 u_1 + \frac{\beta_0}{\tilde\mu_0 \omega} S_0 L_0 u_0  
&=-  \frac{\beta_0}{\tilde\mu_0 \omega} S_0 \partial_\tau e^{inc}_3 .
\end{aligned}
\end{equation*}
We consider now equations \eqref{eq_diri1} and \eqref{eq_diri2}, to obtain
 \begin{equation*}\label{eq_bound3}
\begin{aligned}
\left( D_0 -\frac12 I\right) u_0 - \frac{\tilde\epsilon_1}{\tilde\epsilon_0} S_0 K_1 u_0 - \frac{\beta_1}{\tilde\epsilon_0 \omega} S_0 L_1 v_0 - \frac{\beta_0}{\tilde\epsilon_0 \omega} S_0 L_0 v_0 &= - S_0 \partial_\eta e^{inc}_3 + \frac{\tilde\epsilon_1}{\tilde\epsilon_0} S_0 K_1 e^{inc}_3 ,\\
\left( D_0 -\frac12 I\right) v_0 - \frac{\tilde\mu_1}{\tilde\mu_0} S_0 K_1 v_0 + \frac{\beta_1}{\tilde\mu_0 \omega} S_0 L_1 u_0 + \frac{\beta_0}{\tilde\mu_0 \omega} S_0 L_0 u_0 &= -  \frac{\beta_0}{\tilde\mu_0 \omega} S_0 \partial_\tau e^{inc}_3 - \frac{\beta_1}{\tilde\mu_0 \omega} S_0 L_1 e^{inc}_3.
\end{aligned}
\end{equation*}
The above system in compact form reads
\begin{equation}\label{eq_general1}
(\bm D + \bm K )\, \bm u = \bm b ,
\end{equation}
where
\begin{align*}
\bm D &= \begin{pmatrix}
D_0 -\tfrac12 I & 0 \\ 0 & D_0 -\tfrac12 I
\end{pmatrix},  \\ 
\bm K &= \begin{pmatrix}
- \tfrac{\tilde\epsilon_1}{\tilde\epsilon_0} S_0 K_1 & - \frac1{\tilde\epsilon_0 \omega} S_0 (\beta_1 L_1 + \beta_0 L_0 )  \\ \frac{1}{\tilde\mu_0 \omega} S_0 (\beta_1 L_1 + \beta_0 L_0 )  & - \frac{\tilde\mu_1}{\tilde\mu_0} S_0 K_1
\end{pmatrix}, \\
\bm u &= \begin{pmatrix}
\left. u_0 \right|_\Gamma \\ \left. v_0 \right|_\Gamma
\end{pmatrix} , \quad
\bm b = \begin{pmatrix}
- S_0 \partial_\eta  + \frac{\tilde\epsilon_1}{\tilde\epsilon_0} S_0 K_1 \\ -  \frac1{\tilde\mu_0 \omega} S_0 ( \beta_0 \partial_\tau  + \beta_1 L_1 )
\end{pmatrix} e^{inc}_3 .
\end{align*}

We assume that $\Gamma$ is of class $C^{2,\alpha}, \, 0<\alpha \leq 1 .$ We know that $D_0 : H^{-1/2} (\Gamma) \rightarrow H^{-1/2} (\Gamma)$ is compact, thus  $(D_0 - \frac12 I)^{-1}: H^{-1/2} (\Gamma) \rightarrow H^{-1/2} (\Gamma)$ is bounded, if $\kappa_0^2$ is not an interior Dirichlet eigenvalue. Then, $\bm D : (H^{-1/2} (\Gamma))^2 \rightarrow (H^{-1/2} (\Gamma))^2$ is bounded and \eqref{eq_general1} is transformed to
\begin{equation}\label{eq_general2}
(\bm I + \bm D^{-1}\bm K )\, \bm u = \bm D^{-1} \bm b .
\end{equation}
First, we show that $\bm K$ is compact. We recall that
\begin{align*}
S_0 K_1 &= S_0 \left( NS_1 - \tfrac12 I\right)^{-1} ND_1 , \\
S_0 (\beta_1 L_1 + \beta_0 L_0 ) &= 2 S_0 \left( \beta_1 TD_1 +\beta_0 TD_0 - \beta_0 TS_0 \left( NS_0 + \tfrac12 I\right)^{-1} ND_0 \right) \\
 &-2 \beta_1 S_0 TS_1 \left( NS_1 - \tfrac12 I\right)^{-1} ND_1 ,
\end{align*}
and the following properties, see \cite{ColKre98, CosSte85}, $S_0 : H^{-1/2} (\Gamma) \rightarrow H^{-1/2} (\Gamma)$ is compact,  $ND_j , \, TD_j: H^{1/2} (\Gamma) \rightarrow H^{-1/2} (\Gamma)$ are bounded, $TS_j: H^{-1/2} (\Gamma) \rightarrow H^{-1/2} (\Gamma)$ are bounded, and $\left( NS_j \pm \tfrac12 I\right)^{-1} : H^{-1/2} (\Gamma) \rightarrow H^{-1/2} (\Gamma)$ are bounded due to compactness of $NS_j : H^{-1/2} (\Gamma) \rightarrow H^{-1/2} (\Gamma).$ Then, the operators
\[
S_0 K_1 , \, S_0 ( \beta_1 L_1 + \beta_0 L_0 ) : H^{1/2} (\Gamma) \rightarrow H^{-1/2} (\Gamma) ,
\]
are compact. Hence, $\bm K : (H^{1/2} (\Gamma))^2 \rightarrow (H^{-1/2} (\Gamma))^2$ is also compact resulting to the compactness of $\bm D^{-1} \bm K : (H^{1/2} (\Gamma))^2 \rightarrow (H^{-1/2} (\Gamma))^2 .$

Next we prove the uniqueness of solutions of equation \eqref{eq_general2}. Solvability follows from the Fredholm alternative theorem. Let $(\phi_0 , \psi_0 )^T$ be the solution of the homogeneous form of \eqref{eq_general2}. Then, the potentials
\begin{equation*}
\begin{aligned}
 u_{0,0} (\bm x) &= (D_0 \phi_0 ) (\bm x) - ( S_0 \partial_\eta \phi_{0} ) (\bm x), & \bm x\in \Omega_0 , \\
  v_{0,0} (\bm x) &= (D_0 \psi_0 ) (\bm x) - ( S_0 \partial_\eta \psi_{0}) (\bm x), & \bm x\in \Omega_0 ,
\end{aligned}
\end{equation*}
and
\begin{equation*}
\begin{aligned}
 u_{1,1} (\bm x) &= -(D_1 \phi_1 ) (\bm x) + ( S_1 \partial_\eta \phi_1 ) (\bm x), & \bm x\in \Omega_1 , \\
  v_{1,1} (\bm x) &= -(D_1 \psi_1 ) (\bm x) + ( S_1 \partial_\eta \psi_1 ) (\bm x), & \bm x\in \Omega_1 ,
\end{aligned}
\end{equation*}
solve the homogeneous form of the problem \eqref{maxwell_constant} - \eqref{radiation_constant}. From Theorem \ref{theorem_most}, we have that $u_{0,0} = v_{0,0} = 0,$ in $\Omega_0$ and $u_{1,1} = v_{1,1} = 0,$ in $\Omega_1$ where the densities $\phi_1 , \psi_1$ depend on the solution of the homogeneous case. 

We construct
\begin{equation*}
\begin{aligned}
 u_{0,1} (\bm x) &= (D_0 \phi_0 ) (\bm x) - ( S_0 \partial_\eta \phi_{0} ) (\bm x), & \bm x\in \Omega_1 , \\
  v_{0,1} (\bm x) &= (D_0 \psi_0 ) (\bm x) - ( S_0 \partial_\eta \psi_{0}) (\bm x), & \bm x\in \Omega_1 ,
\end{aligned}
\end{equation*}
and
\begin{equation*}
\begin{aligned}
 u_{1,0} (\bm x) &= -(D_1 \phi_1 ) (\bm x) + ( S_1 \partial_\eta \phi_1 ) (\bm x), & \bm x\in \Omega_0 , \\
  v_{1,0} (\bm x) &= -(D_1 \psi_1 ) (\bm x) + ( S_1 \partial_\eta \psi_1 ) (\bm x), & \bm x\in \Omega_0 .
\end{aligned}
\end{equation*}
Considering the jump relations, at the boundary $\Gamma$ we obtain
\begin{equation}\label{eq_boundcon}
\begin{aligned}
 u_{0,0} -  u_{0,1} &= \phi_0 , & u_{1,1} -  u_{1,0} &= \phi_1 , \\
  v_{0,0} -  v_{0,1} &= \psi_0 , & v_{1,1} -  v_{1,0} &= \psi_1 .
\end{aligned}
\end{equation}
Since $u_{0,0} = u_{1,1} = v_{0,0} = v_{1,1} = 0, \, \phi_0 = \phi_1 ,$ and $\psi_0 = \psi_1 ,$ on $\Gamma ,$ we find
\[
u_{0,1} = u_{1,0} , \quad v_{0,1} = v_{1,0} , \quad \bm x \in \Gamma .
\]
Similarly, we can rewrite the other two boundary conditions of \eqref{boundary_homo} for those fields taking the differences of the normal and tangential derivatives as $\bm x \rightarrow \Gamma .$ Thus, we see that $u_{0,1}, v_{0,1}, u_{1,0}$ and $v_{1,0}$ solve the homogeneous problem but with $\kappa_1$ and $\kappa_0$ interchanged and from Theorem \ref{theorem_most} we get also $u_{0,1} = v_{0,1} = 0,$ on $\Gamma$ and hence $\phi_0 = \psi_0 = 0$ from \eqref{eq_boundcon}.
\end{proof}

In order to handle the hypersingularity of the operators $TD_j$ we work in a similar way as Mitzner \cite{Mit66} derived the Maue's formula \cite{Mau49} of the hypersingular operator $ND_j ,$ namely
\begin{equation}\label{maue}
(ND_j f) (\bm x)   = \left( TS_j  \frac{\partial f}{\partial \tau } \right)  (\bm x) + \kappa_j^2 \, \n (\bm x) \cdot (S_j \,\n f ) (\bm x) , \quad \bm x \in \Gamma .
\end{equation}
This transformation reduces the hypersingularity to singularity of Cauchy type (first term) and to  a weak singularity  (second term). 

\begin{theorem}
Let $f \in C^{1,\alpha} (\Gamma).$ The hypersingular operator $TD_j$ can be transformed to
\begin{equation}\label{maue2}
(TD_j f) (\bm x)   = -\left( NS_j  \frac{\partial f}{\partial \tau } \right)  (\bm x) + \kappa_j^2 \, \ta (\bm x) \cdot (S_j \,\n f ) (\bm x) ,  \quad \bm x \in \Gamma .
\end{equation}
\end{theorem}

\begin{proof}
We recall equation \eqref{eq_green0}. Applying Green's first theorem to $\bm v \cdot \nabla u_1 , \bm v$ arbitrary constant vector, and $\Phi_1 ,$  using that
\[
\Delta \Phi_1 (\bm x,\bm y)+ \kappa^2_1 \Phi_1 (\bm x,\bm y)= - \delta (\bm x - \bm y ),
\]
  yields
  \begin{equation} \label{eq_green1}
  \begin{aligned}
\int_{\Omega_1} \nabla_y (\bm v \cdot \nabla_y u_1 (\bm y) ) \cdot \nabla_y \Phi_1 (\bm x,\bm y) d \bm y - \kappa_1^2 \int_{\Omega_1} \Phi_1 (\bm x,\bm y) \, \bm v \cdot \nabla_y u_1 (\bm y) d \bm y  \\
= \bm v \cdot \nabla_x u_1 (\bm x) +\int_\Gamma \frac{\partial \Phi_1}{\partial n (\bm y)} (\bm x,\bm y) \bm v \cdot \nabla_y u_1 (\bm y) ds (\bm y) .
\end{aligned}
  \end{equation}
The first integral can be transformed to
\begin{align*}
\int_{\Omega_1} \nabla (\bm v \cdot \nabla u_1 ) \cdot \nabla \Phi_1 \, d \bm y &= - \int_{\Omega_1} \left(  \kappa_1^2   u_1  \bm v + \nabla \times (\bm v \times \nabla u_1 )  \right) \cdot \nabla \Phi_1 \, d \bm y  - \kappa_1^2   \bm v \cdot \int_{\Omega_1}  u_1 \nabla \Phi_1 \, d \bm y \\
&- \bm v \cdot \int_\Gamma (\n \times \nabla \Phi_1 ) \times \nabla u_1 \, ds (\bm y ) .
\end{align*} 

Then, \eqref{eq_green1} reads
\begin{equation*}
  \begin{aligned}
 \bm v \cdot \left(   \kappa_1^2    \int_{\Omega_1} \nabla_y (\Phi_1  \, u_1 ) d \bm y 
  + \int_\Gamma (\n \times \nabla_y \Phi_1 ) \times \nabla_y u_1 \, ds (\bm y ) + \int_\Gamma \frac{\partial \Phi_1}{\partial n_y }   \nabla_y u_1 \, ds (\bm y) \right) 
    = -\bm v \cdot    \nabla_x u_1 .
\end{aligned}
\end{equation*}

Using some vector identities and suppressing the inner products with $\bm v$ (holds for any vector), we end up to
\begin{equation*}
- \nabla_x u_1  = \int_\Gamma \left(  - \nabla_y \Phi_1  \times (\n_y \times  \nabla_y u_1 ) +  \kappa_1^2 \Phi_1 u_1 \n_y + \nabla_y \Phi_1 \frac{\partial u_1}{\partial n_y } \right)  ds (\bm y) , 
\end{equation*}
for $\bm x \in \Omega_1 .$ We multiple this equation with $\n (\bm x)$ (inner product) and considering the limit as $\bm x$ approaches the boundary $\Gamma$ from inside and the corresponding jump relations, we obtain \cite{Mit66}
\begin{equation}\label{eq_green3}
\begin{aligned}
 - \frac12\frac{\partial }{\partial n }  u_1  &= \int_\Gamma \left(  (\n_x \times \nabla_x \Phi_1  ) \cdot (\n_y \times  \nabla_y u_1 ) +  \kappa_1^2 \Phi_1 u_1 (\n_x \cdot \n_y ) \right)  ds (\bm y) \\
 &- \int_\Gamma \frac{\partial \Phi_1}{\partial n_x } \frac{\partial u_1}{\partial n_y }   ds (\bm y) , \quad \bm x \in \Gamma .
 \end{aligned}
\end{equation}

Now, equating the above equation and the normal derivative of \eqref{eq_green0} as $\bm x \rightarrow \Gamma$ we obtain \eqref{maue} in $\R^2 .$  

We take the tangential derivative of \eqref{eq_green0} and considering the jump relations we get
\begin{equation*}
 - \frac{\partial}{\partial \tau } u_1 (\bm x) = (TD_1 u_1 ) (\bm x) - \frac12\frac{\partial u_1 }{\partial \tau }  (\bm x)  - \left(  TS_1 \frac{\partial u_1}{\partial n} \right)  (\bm x), \quad \bm x \in \Gamma .
\end{equation*}

We replace $\n (\bm x)$ by $\ta (\bm x)$ in \eqref{eq_green3} (considering the appropriate jump relations) and restricting ourselves in $\R^2 ,$ we have
\begin{equation*}
\begin{aligned}
 - \frac12\frac{\partial }{\partial \tau  }  u_1  &= \int_\Gamma \left( - (\n_x \cdot  \nabla_x \Phi_1  ) (\ta_y \cdot  \nabla_y u_1 ) +  \kappa_1^2 \Phi_1 u_1 (\ta_x \cdot \n_y ) \right)  ds (\bm y) - \int_\Gamma \frac{\partial \Phi_1}{\partial \tau_x } \frac{\partial u_1}{\partial n_y }   ds (\bm y) , \,\, \bm x \in \Gamma .
\end{aligned}
\end{equation*}

Observing the last two equations, we obtain \eqref{maue2}, the equivalent of the Maue's  formula for the tangential derivative of the double-layer potential which also reduces the hypersingularity of the potential. 
\end{proof}

\section{Numerical results} In this section we present numerical examples by implementing the proposed method. We use quadrature rules to integrate the singularities considering trigonometric interpolation. Regarding the convergence and the error analysis of the quadrature formulas, we refer the reader to \cite{Kre99} for the weakly singular operators and to \cite{Kre95} for the hypersingular. We solve the system of integral equations considering these rules by the Nystr\"om method.

We assume the following parametrization for the boundary
\[
\Gamma = \{ \bm z (t) = (z_1 (t), z_2 (t)) : t \in [0,2\pi]\},
\]
where $\bm z : \R \rightarrow \R^2$ is a $C^2$-smooth, $2\pi$-periodic and counter-clockwise oriented parametrization. We assume in addition that $\bm z$  is injective in $[0,2\pi),$ that is $\bm z'  (t) \neq 0,$ for all $t\in [0,2\pi].$

Now, we transform the operators in \eqref{single_double} and their derivatives, see \eqref{jump_relations}, into their parametric forms
\begin{subequations}
\begin{align*}
(S_j \psi ) (t) &= \int_{0}^{2\pi}  M^{S_j} (t,s)   \psi (s)  ds ,  \\
(D_j \psi ) (t) &= \int_{0}^{2\pi}  M^{D_j} (t,s)    \psi (s)  ds ,  \\
(NS_j \psi ) (t) &= \int_{0}^{2\pi} M^{NS_j} (t,s)     \psi (s)   ds ,  \\
(TS_j \psi ) (t) &= \int_{0}^{2\pi}  M^{TS_j} (t,s)   \psi (s) ds ,
\end{align*}
\end{subequations}
and the special forms
\begin{subequations}
\begin{align} 
(ND_j \psi ) (t) &= \frac1{|\bm z' (t)|}\int_{0}^{2\pi} \left[  \frac1{4\pi} \cot \left( \frac{s -t}2\right) \frac{\partial \psi}{\partial s} (s)  -  M^{ND_j} (t,s)   \psi (s) \right]  ds  \label{eq_ND} \\ 
&+ \kappa_j^2 \int_{0}^{2\pi} (\n (t) \cdot \n (s))  M^{S_j} (t,s)   \psi (s)  ds , \nonumber \\ 
(TD_j \psi ) (t) &= \frac1{|\bm z' (t)|}\int_{0}^{2\pi}   M^{TD_j} (t,s)   \psi (s) ds \label{eq_TD} \\ 
&+  \kappa_j^2 \int_{0}^{2\pi} (\ta (t) \cdot \n (s))  M^{S_j} (t,s)   \psi (s)  ds \nonumber,
\end{align}
\end{subequations}
for $t \in [0,2\pi], \, j=0,1 ,$ and $\psi (t) =  f (\bm z (t)), \, \n (t) = \n (\bm z(t)), \, \ta (t) = \ta (\bm z(t)),$ where 
\begin{equation}\label{eq_kernels}
\begin{aligned}
M^{S_j} (t,s)  &=   \frac{i}4 H_0^{(1)} (\kappa_j |\bm r (t,s)|) |\bm z' (s)| ,  \\
M^{D_j} (t,s) &= \frac{i\kappa_j}4 \frac{\n (s) \cdot \bm r (t,s)}{|\bm r (t,s)|} H_1^{(1)} (\kappa_j |\bm r (t,s)|) |\bm z' (s)|,  \\
M^{NS_j} (t,s) &= - \frac{i\kappa_j}4 \frac{\n (t) \cdot \bm r (t,s)}{|\bm r (t,s)|} H_1^{(1)} (\kappa_j |\bm r (t,s)|) |\bm z' (s)|,  \\
M^{TS_j} (t,s)  &=  - \frac{i\kappa_j}4 \frac{\ta (t) \cdot \bm r (t,s)}{|\bm r (t,s)|} H_1^{(1)} (\kappa_j |\bm r (t,s)|) |\bm z' (s)|, \\
M^{ND_j} (t,s) &= \frac{i}4 M (t,s) \left[ \kappa_j^2 H_0^{(1)} (\kappa_j |\bm r (t,s)|) - 2 \kappa_j \frac{H_1^{(1)} (\kappa_j |\bm r (t,s)|)}{|\bm r (t,s)|} \right] \\
&+ \frac{i\kappa_j}4 \frac{\bm z' (t) \cdot \bm z' (s)}{|\bm r (t,s)|} H_1^{(1)} (\kappa_j |\bm r (t,s)|) + \frac1{8\pi} \sin^{-2} \left( \frac{t - s }2\right),\\
M^{TD_j} (t,s) &=  \frac{i}4 M^J (t,s) \left[ \kappa_j^2 H_0^{(1)} (\kappa_j |\bm r (t,s)|) - 2 \kappa_j \frac{H_1^{(1)} (\kappa_j |\bm r (t,s)|)}{|\bm r (t,s)|} \right] \\
&+ \frac{i\kappa_j}4 \frac{\bm J \,\bm z' (t) \cdot \bm z' (s)}{|\bm r (t,s)|} H_1^{(1)} (\kappa_j |\bm r (t,s)|),
\end{aligned}
\end{equation}
with $\bm r (t,s) = \bm z (t)-\bm z(s),$ and
\begin{align*}
M (t,s) &= \frac{(\bm z' (t) \cdot  \bm r (t,s))(\bm z' (s) \cdot  \bm r (t,s))}{|\bm r (t,s)|^2}, \\
M^J (t,s) &= \frac{( \bm J \,\bm z' (t) \cdot  \bm r (t,s))(\bm z' (s) \cdot  \bm r (t,s))}{|\bm r (t,s)|^2} . \\
\end{align*}
Here, we have used the formulas $H_1^{(1)}(t) =  - H_0^{(1)'}(t)$ and $H_1^{(1)'}(t) =   H_0^{(1)}(t) - \tfrac1t H_1^{(1)}(t).$ The form in equation \eqref{eq_ND} is based on \eqref{maue}, derived by Kress \cite{Kre95} and improved in \cite{Kre14}. The derivation of \eqref{eq_TD} is easier. Namely, we define $\partial_{t_J} := \bm J \, \bm z' \partial_z$ and
\[
M^{TD_j} (t,s) := \frac{\partial^2}{\partial t_J \partial s} M^{S_j} (t,s),
\]
then
\begin{align*}
-\left( NS_j  \frac{\partial f }{\partial \tau } \right)  (\bm z(t)) &= -\int_{0}^{2\pi} 
\frac{\bm J \, \bm z' (t)}{|\bm z' (t)|} \frac{\partial}{\partial z (t)} M^{S_j} (t,s) \frac{\bm z' (s)}{|\bm z' (s)|} \frac{\partial \psi}{\partial z(s)} (s)  |\bm z' (s)| ds \\
&= - \frac1{|\bm z' (t)|} \int_{0}^{2\pi}  \frac{\partial}{\partial t_J} M^{S_j} (t,s) \frac{\partial }{\partial s} \psi (s) ds \\
&=  \frac1{|\bm z' (t)|} \int_{0}^{2\pi}  \frac{\partial^2}{\partial t_J \partial s} M^{S_j} (t,s) \psi (s) ds 
\end{align*}
and \eqref{eq_TD} follows by simply adding the parametrized form  of the second term in the right-hand side of \eqref{maue2}. The kernels in \eqref{eq_kernels} admit the decomposition
\[
M^k (t,s) = M_1^k (t,s) \ln \left( 4 \sin^2 \left(\frac{t - s}2\right) \right) + M_2^k (t,s) ,
\]
for $k = S_j ,D_j , NS_j , ND_j , TD_j$ where $M_1^k$ and $M_2^k$ are analytic, due to logarithmic singularity of the functions at $t = s$. The case of $M^{TS_j}$  has to be treated differently because of the Cauchy type singularity of the kernel as $t = s.$ Thus, we split the kernel as
\begin{align*}
M^{TS_j} (t,s) &= M_1^{TS_j} (t,s) \ln \left( 4 \sin^2 \left(\frac{t - s}2\right) \right) + \frac1{4\pi} \cot \left( \frac{s-t}2\right) +  M_2^{TS_j} (t,s).
\end{align*}

The kernels $M_1^k$ are defined for $ t \neq s$ by, see \cite{Kre95}
\begin{align*}
M_1^{S_j} (t,s) &= 
      -\dfrac1{4\pi} J_0 (\kappa_j |\bm r (t,s)|) |\bm z' (s)|,\\
   M_1^{D_j} (t,s) &= 
   - \dfrac{ \kappa_j}{4\pi} \dfrac{\n (s) \cdot \bm r (t,s)}{|\bm r (t,s)|} J_1 (\kappa_j |\bm r (t,s)|) |\bm z' (s)| ,  \\
    M_1^{NS_j} (t,s) &= 
   \dfrac{ \kappa_j}{4\pi} \dfrac{\n (t) \cdot \bm r (t,s)}{|\bm r (t,s)|} J_1 (\kappa_j |\bm r (t,s)|) |\bm z' (s)| , \\
     M_1^{TS_j} (t,s)  &=   \frac{ \kappa_j}{4\pi} \frac{\ta (t) \cdot \bm r (t,s)}{|\bm r (t,s)|} J_1 (\kappa_j |\bm r (t,s)|) |\bm z' (s)|, \\
      M_1^{ND_j} (t,s) &= 
  -\frac1{4\pi} M (t,s) \left[ \kappa_j^2 J_0 (\kappa_j |\bm r (t,s)|) - 2 \kappa_j \frac{J_1 (\kappa_j |\bm r (t,s)|)}{|\bm r (t,s)|} \right] \\ 
  &- \frac{ \kappa_j}{4\pi} \frac{\bm z' (t) \cdot \bm z' (s)}{|\bm r (t,s)|} J_1 (\kappa_j |\bm r (t,s)|) , \\
M_1^{TD_j} (t,s) &=  -\frac1{4\pi} M^J (t,s) \left[ \kappa_j^2 J_0 (\kappa_j |\bm r (t,s)|) - 2 \kappa_j \frac{J_1 (\kappa_j |\bm r (t,s)|)}{|\bm r (t,s)|} \right] \\
&- \frac{ \kappa_j}{4\pi} \frac{\bm J \,\bm z' (t) \cdot \bm z' (s)}{|\bm r (t,s)|} J_1 (\kappa_j |\bm r (t,s)|) ,
\end{align*}
with diagonal terms
\begin{equation*}
\begin{aligned}
M_1^{S_j} (t,t) &=  -\dfrac1{4\pi} |\bm z' (t)|  , &
   M_1^{D_j} (t,t) &= 0, \\
    M_1^{NS_j} (t,t) &= 0, 
       &   M_1^{TS_j} (t,t)  &=  0 
  , \\
M_1^{ND_j} (t,t) &= -\frac{\kappa_j^2}{8\pi} |\bm z' (t)|^2 ,  &
M_1^{TD_j} (t,t) &= 0,
\end{aligned}
\end{equation*}
where $J_0 , \, J_1$  are the Bessel functions of order zero and one, respectively. 
The kernels $M^k_2$, for $t \neq s ,$ are given by
\[
M_2^k (t,s)  = M^k (t,s) - M_1^k (t,s) \ln \left( 4 \sin^2 \left(\frac{t - s}2\right) \right)
\]
and
\begin{align*}
M_2^{TS_j} (t,s) &= M^{TS_j} (t,s) - M_1^{TS_j} (t,s) \ln \left( 4 \sin^2 \left(\frac{t - s}2\right) \right) - \frac1{4\pi} \cot \left( \frac{s-t}2\right),
\end{align*}
with diagonal terms
\begin{equation*}
\begin{aligned}
M_2^{S_j} (t,t) &= \left[  \frac{i}4 -\frac{C}{2\pi} - \frac1{2\pi} \ln \left( \frac{\kappa_j}2 |\bm z' (t)|\right) \right]  |\bm z' (t)| , \\
   M_2^{D_j} (t,t) &= \frac1{4\pi} \frac{\n (t) \cdot \bm z''(t)}{|\bm z' (t)|}, \\
    M_2^{NS_j} (t,t) &= \frac1{4\pi} \frac{\n (t) \cdot \bm z''(t)}{|\bm z' (t)|}, \\
     M_2^{TS_j} (t,t)  &=  -\frac1{4\pi} \frac{\ta  (t) \cdot \bm z''(t)}{|\bm z' (t)|}, \\
         M_2^{ND_j} (t,t) &= \left[  \pi i-1 - 2C - 2\ln \left( \frac{\kappa_j}2 |\bm z' (t)|\right)\right] \frac{\kappa_j^2}{8\pi} |\bm z' (t)|^2
   +  \frac1{24\pi} \\
   &+ \frac{(\bm z' (t) \cdot \bm z'' (t) )^2}{4\pi |\bm z' (t)|^4} - \frac{\bm z' (t) \cdot \bm z''' (t) }{12\pi |\bm z' (t)|^2} - \frac{ |\bm z'' (t) |^2}{8\pi |\bm z' (t)|^2} ,\\
M_2^{TD_j} (t,t) &=   - \frac{(\bm z' (t) \cdot \bm z'' (t) ) (\bm J \,\bm z' (t) \cdot \bm z'' (t))}{4\pi |\bm z' (t)|^4} + \frac{\bm J \, \bm z' (t) \cdot \bm z''' (t) }{12\pi |\bm z' (t)|^2} ,
\end{aligned}
\end{equation*}
where $C$ is the Euler's constant. For the last approximation, we used same arguments as in the case of $M_2^{ND_j}.$ 

Considering the equidistant points $t_j = j \pi /n, \, j=0,...,2n-1 ,$ we use the trapezoidal rule to approximate the operators with smooth kernel
\[
\int_0^{2\pi}  \psi (s) ds \approx \frac{\pi}n \sum_{j=0}^{2n-1} \psi (t_j),
\]
and the following quadrature rules for the singular kernels
\begin{equation*}
\begin{aligned}
\int_0^{2\pi} \ln \left( 4 \sin^2 \left(\frac{t - s}2\right) \right) \psi (s) ds &\approx  \sum_{j=0}^{2n-1} R_j^{(n)} (t) \psi (t_j), \\
\frac1{4\pi} \int_0^{2\pi} \cot \left( \frac{s-t}2\right) \frac{\partial}{\partial s}\psi (s) ds &\approx  \sum_{j=0}^{2n-1} T_j^{(n)} (t) \psi (t_j), \\
 \int_0^{2\pi} \cot \left( \frac{s-t}2\right) \psi (s) ds &\approx  \sum_{j=0}^{2n-1} S_j^{(n)} (t) \psi (t_j), 
\end{aligned}
\end{equation*}
with weights
\begin{equation*}
\begin{aligned}
R_j^{(n)} (t) &= -\frac{2\pi}n \sum_{m=1}^{n-1} \frac1{m} \cos \left( m (t - t_j)\right) - \frac{\pi}{n^2} \cos (n(t - t_j)), \\
T_j^{(n)} (t) &= -\frac1{2n} \sum_{m=1}^{n-1} m \cos \left( m (t - t_j)\right) - \frac14 \cos (n(t - t_j)), \\
S_j^{(n)} (t) &= \frac{\pi}n \left[ 1 - (-1)^j \cos (nt) \right] \cot \left( \frac{t_j -t}2\right), \quad t \neq t_j. 
\end{aligned}
\end{equation*}

Then, the system \eqref{eq_general1}, or similarly \eqref{eq_general2}, considering the above parametric forms of the integral operators and the quadrature rules, is transformed to a linear system by applying the Nystr\"om method. 

To illustrate the efficiency of our method, we consider two different cases. In the first 
example, motivated by \cite{WanNak12}, we construct a model where the scattered fields can be analytically computed and in the second one we consider the scattering of obliquely incident waves.

In both examples, the parametrization of the obstacle is given by
\[
\bm z (t) = (2 \cos t + 1.5 \cos 2t -1 ,\, 2.5 \sin t), \quad t\in [0,2\pi].
\]

\subsection{Example with analytic solution}
We consider four arbitrary points $\bm z_1 , \bm z_2 \in \Omega_1$ and $\bm z_3 , \bm z_4 \in \Omega_0 $ and we define the boundary functions $f_k , \, k=1,2,3,4$  by
\begin{align*}
f_1 &= H_0^{(1)} (\kappa_1 |\bm r_3 (\bm x) |) - H_0^{(1)} (\kappa_0 |\bm r_1 (\bm x) |) , \\
f_2 &= -\tilde\mu_1 \omega \kappa_1 H_1^{(1)} (\kappa_1 |\bm r_4 (\bm x) |) \frac{\n (\bm x) \cdot \bm r_4 (\bm x)}{|\bm r_4 (\bm x) |} 
 - \beta_1 \kappa_1 H_1^{(1)} (\kappa_1 |\bm r_3 (\bm x) |) \frac{\ta (\bm x) \cdot \bm r_3 (\bm x)}{|\bm r_3 (\bm x) |}\\
&+\tilde\mu_0\omega \kappa_0 H_1^{(1)} (\kappa_0 |\bm r_2 (\bm x) |) \frac{\n (\bm x) \cdot \bm r_2 (\bm x)}{|\bm r_2 (\bm x) |} +\beta_0 \kappa_0 H_1^{(1)} (\kappa_0 |\bm r_1 (\bm x) |) \frac{\ta (\bm x) \cdot \bm r_1 (\bm x)}{|\bm r_1 (\bm x) |} ,\\
f_3 &= H_0^{(1)} (\kappa_1 |\bm r_4 (\bm x) |) -  H_0^{(1)} (\kappa_0 |\bm r_2 (\bm x) |) ,\\
f_4 &= -\tilde\epsilon_1 \omega \kappa_1 H_1^{(1)} (\kappa_1 |\bm r_3 (\bm x) |) \frac{\n (\bm x) \cdot \bm r_3 (\bm x)}{|\bm r_3 (\bm x) |} + \beta_1 \kappa_1 H_1^{(1)} (\kappa_1 |\bm r_4 (\bm x) |) \frac{\ta (\bm x) \cdot \bm r_4 (\bm x)}{|\bm r_4 (\bm x) |}\\
&+\tilde\epsilon_0\omega \kappa_0 H_1^{(1)} (\kappa_0 |\bm r_1 (\bm x) |) \frac{\n (\bm x) \cdot \bm r_1 (\bm x)}{|\bm r_1 (\bm x) |} -\beta_0 \kappa_0 H_1^{(1)} (\kappa_0 |\bm r_2 (\bm x) |) \frac{\ta (\bm x) \cdot \bm r_2 (\bm x)}{|\bm r_2 (\bm x) |} ,
\end{align*}
where $\bm r_k (\bm x) = \bm x - \bm z_k .$ Then, the fields 
\begin{equation}\label{eq_exact}
\begin{aligned}
u_0 (\bm x) &= H_0^{(1)} (\kappa_0 |\bm x- \bm z_1 |), &  v_0 (\bm x) &= H_0^{(1)} (\kappa_0 |\bm x- \bm z_2 |), \quad \bm  x \in \Omega_0, \\
u_1 (\bm x) &= H_0^{(1)} (\kappa_1 |\bm x- \bm z_3 |), &  v_1 (\bm x) &= H_0^{(1)} (\kappa_1 |\bm x- \bm z_4 |), \quad \bm  x \in \Omega_1 ,
\end{aligned}
\end{equation}
solve the following problem
\begin{equation*}
\begin{aligned}
\Delta u_0 + \kappa^2_0 \, u_0 &= 0, &
\Delta v_0 + \kappa^2_0 \, v_0 &= 0, & \bm  x \in \Omega_0, \\
\Delta u_1 + \kappa^2_1 \, u_1 &= 0, &
\Delta v_1 + \kappa^2_1 \, v_1 &= 0, & \bm  x \in \Omega_1, 
\end{aligned}
\end{equation*}
with boundary conditions
\begin{equation*}
\begin{aligned}
u_1 &= u_0 + f_1, & \bm  x\in \Gamma , \\
\tilde\mu_1 \omega \frac{\partial v_1}{\partial n }  + \beta_1 \frac{\partial u_1}{\partial \tau } &= \tilde\mu_0 \omega \frac{\partial v_0}{\partial n }  + \beta_0 \frac{\partial u_0}{\partial \tau }+ f_2, & \bm  x\in \Gamma, \\ 
v_1 &= v_0 +f_3 , & \bm x\in \Gamma , \\
\tilde\epsilon_1 \omega \frac{\partial u_1 }{\partial n }  - \beta_1 \frac{\partial v_1}{\partial \tau } &= \tilde\epsilon_0 \omega \frac{\partial u_0}{\partial n }   - \beta_0 \frac{\partial v_0}{\partial \tau } +f_4, & \bm x\in \Gamma ,
\end{aligned}
\end{equation*}
and the radiation conditions
\begin{equation*}
\begin{aligned}
\lim_{r \rightarrow \infty} \sqrt{r} \left( \frac{\partial u_0}{\partial r} - i\kappa_0 u_0 \right) =0 , \quad
 \lim_{r \rightarrow \infty} \sqrt{r} \left( \frac{\partial v_0}{\partial r} - i\kappa_0 v_0 \right) =0 . \\
\end{aligned}
\end{equation*}

\begin{figure}
\begin{center}
\includegraphics[scale=0.3]{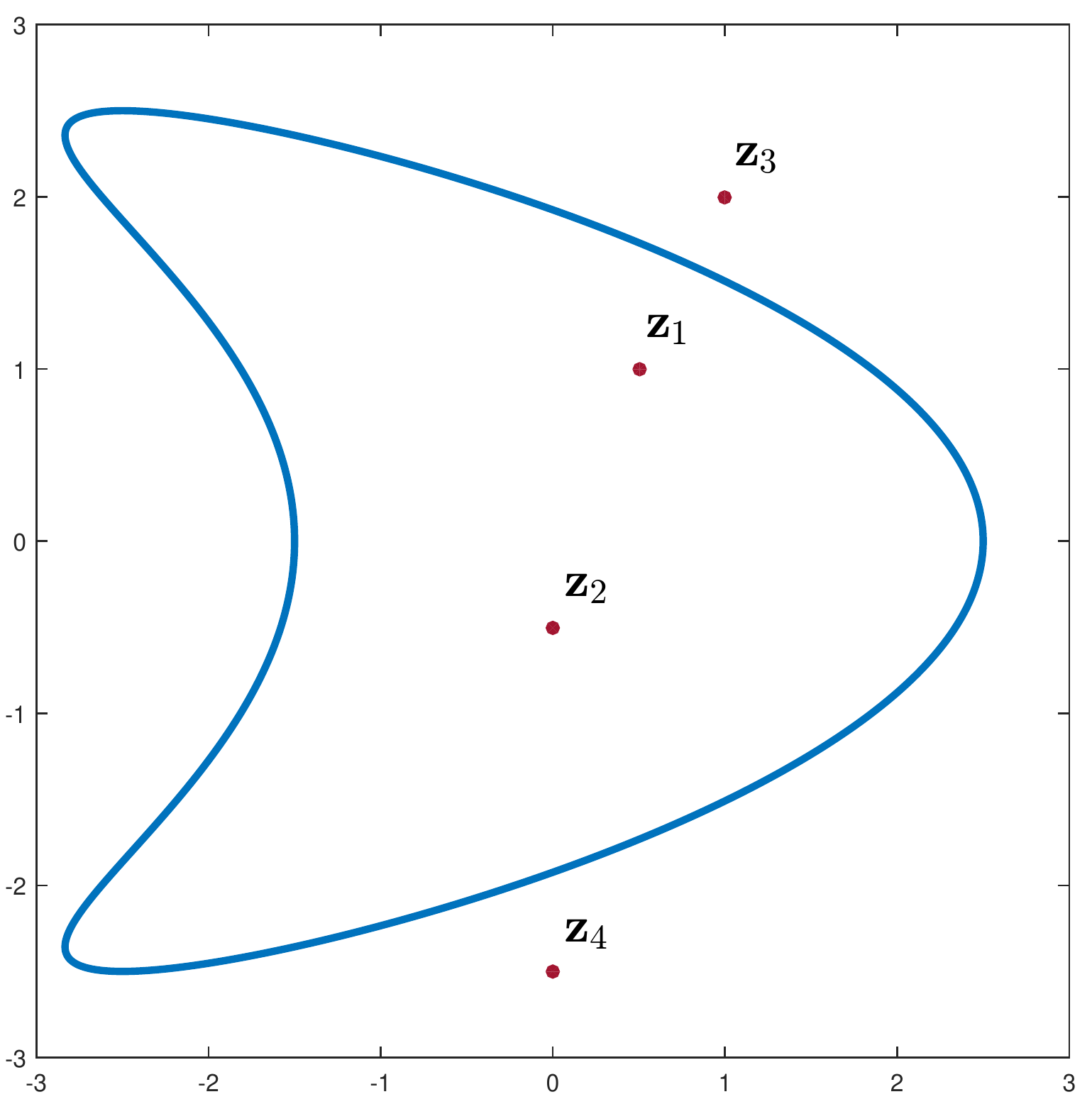}
\caption{The parametrization of the boundary $\Gamma$ and the source points.}\label{Fig2}
\end{center}
\end{figure}

For this problem, we can derive again a system as \eqref{eq_general1}, where now $\bm b$ is replaced by
\begin{equation*}
\textbf{b}_f = \begin{pmatrix}
- \tfrac1{\tilde\epsilon_0 \omega} S_0 f_4  + \tfrac{\tilde\epsilon_1}{\tilde\epsilon_0} S_0 K_1 f_1 + \tfrac{\beta_1}{\tilde\epsilon_0 \omega} S_0 L_1  f_3 \\ 
- \tfrac1{\tilde\mu_0 \omega} S_0 f_2  + \tfrac{\tilde\mu_1}{\tilde\mu_0} S_0 K_1 f_3 -\tfrac{\beta_1}{\tilde\mu_0 \omega} S_0 L_1  f_1
\end{pmatrix} .
\end{equation*}

Given \eqref{eq_exact} and the asymptotic behaviour of the Hankel function \cite{ColKre98}, we know that the far field patterns of $u_0$ and $v_0$ are given by
\begin{equation}\label{far_exact}
u^\infty_0 (\bm{\hat{x}}) = \frac{-4 ie^{i\pi /4}}{\sqrt{8\pi \kappa_0}} e^{-i\kappa_0 \bm{\hat{x}} \cdot \bm z_1}, \quad
v^\infty_0 (\bm{\hat{x}}) = \frac{-4 ie^{i\pi /4}}{\sqrt{8\pi \kappa_0}} e^{-i\kappa_0 \bm{\hat{x}} \cdot \bm z_2}, \quad \bm{\hat{x}} \in \SQ,
\end{equation}
where $\SQ$ is the unit ball. Numerically, the far field patterns are given by
\begin{equation}\label{far_comp}
\begin{aligned}
u^\infty_0 (\bm{\hat{x}}) &= \frac{e^{i\pi /4}}{\sqrt{8\pi \kappa_0}} \int_0^{2\pi} e^{-i\kappa_0 \bm{\hat{x}} \cdot \bm z (s)} \left[ -i\kappa_0  (\bm{\hat{x}} \cdot \n (s)) \varphi_0 (s)  - (K_0 \varphi_0) (s) \right] |\bm z' (s)| ds, \\
v^\infty_0 (\bm{\hat{x}}) &= \frac{e^{i\pi /4}}{\sqrt{8\pi \kappa_0}} \int_0^{2\pi} e^{-i\kappa_0 \bm{\hat{x}} \cdot \bm z (s)} \left[ -i\kappa_0  (\bm{\hat{x}} \cdot \n (s)) \psi_0 (s) - (K_0 \psi_0) (s) \right] |\bm z' (s)| ds ,
\end{aligned}
\end{equation}
where $\bm{\varphi } := (\varphi_0 , \, \psi_0)^T$ solves
\[
(\bm D + \bm K )\, \bm{\varphi } = \bm b_f .
\]
Here, we have used the representations \eqref{eq_greenD0} for the exterior fields and the asymptotics of the Hankel function. The operator $K_0$ is given by \eqref{operqtorK0}.

We consider the points $\bm z_1 = (0.5 ,\, 1)$ and $\bm z_2 = (0 ,\, -0.5)$ in $\Omega_1$ and the  points $\bm z_3 = (1 ,\, 2)$ and $\bm z_4 = (0 ,\, -2.5)$ in $\Omega_0$, see Figure \ref{Fig2}. We set $\omega =1$ and $n=32.$ The exact values \eqref{far_exact} and the reconstructed \eqref{far_comp} for $(\epsilon_1 , \mu_1) = (3, \,2)$ and $(\epsilon_0 , \mu_0) = (1, \,1)$ are presented in Figure \ref{Fig3}. The results are presented for $\theta = \pi/3 .$ In Table \ref{table1}, we provide the absolute errors of the far field patterns for different values of $n$ and $t.$ 

As a general comment we could say that the reconstructions are accurate and illustrate the feasibility of the proposed method. However, the convergence is slower compared to the impedance cylinder case \cite{WanNak12}. The main reason is the complexity of the matrix $\bm K$ involving the product of four operators: $S_0 \, TS_j \left( NS_j + \tfrac12 I\right)^{-1} ND_j , j=0,1$ resulting to an increase of the condition number. As $\theta \rightarrow \pi/2 ,$ the results improve considerably.

\begin{figure}[t]
\begin{center}
\includegraphics[scale=0.5]{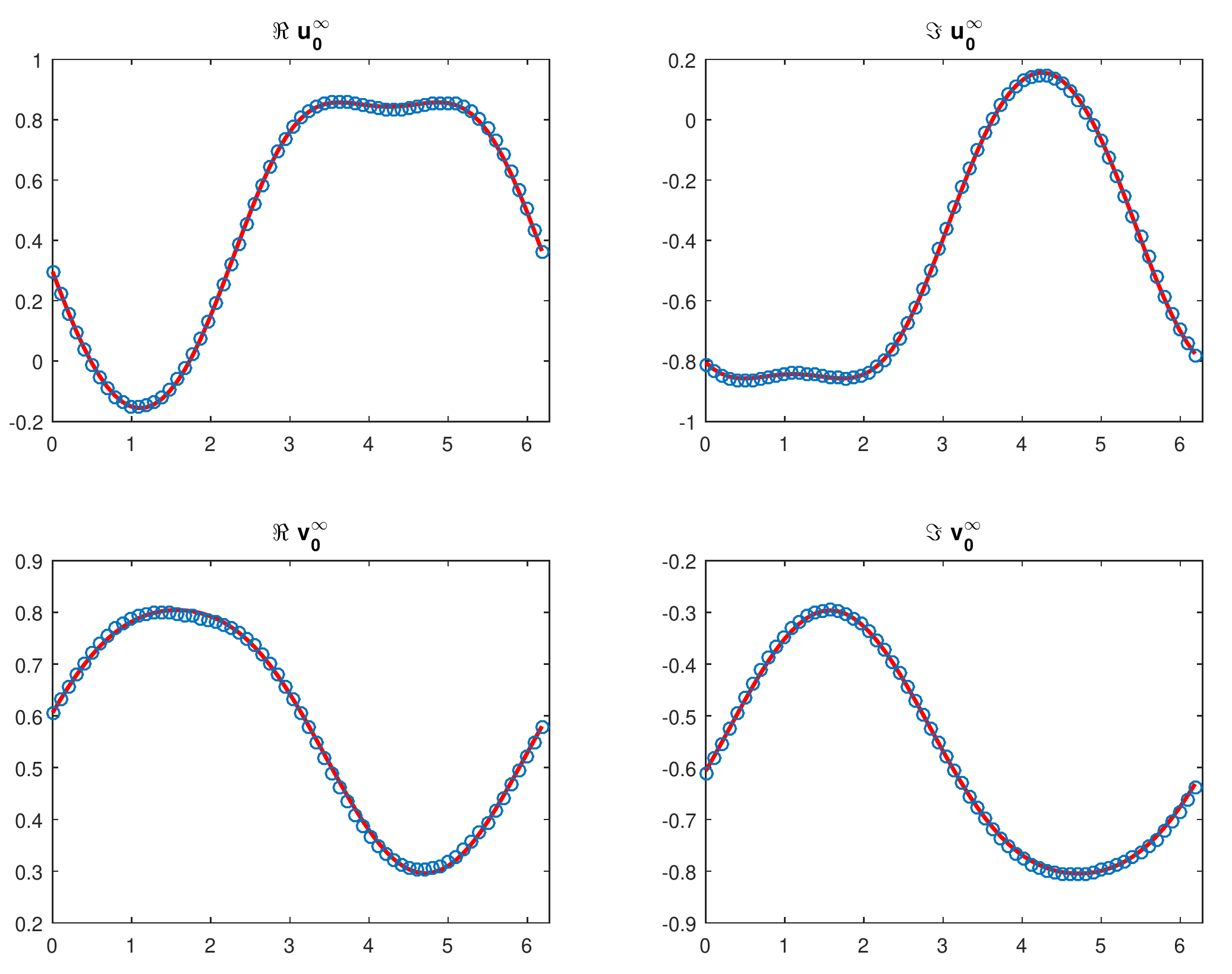}
\caption{The far field patterns: Reconstructed (blue open circles) and exact (red solid line). }\label{Fig3}
\end{center}
\end{figure}

\begin{table}[t]
\begin{center}
\includegraphics[scale=0.82]{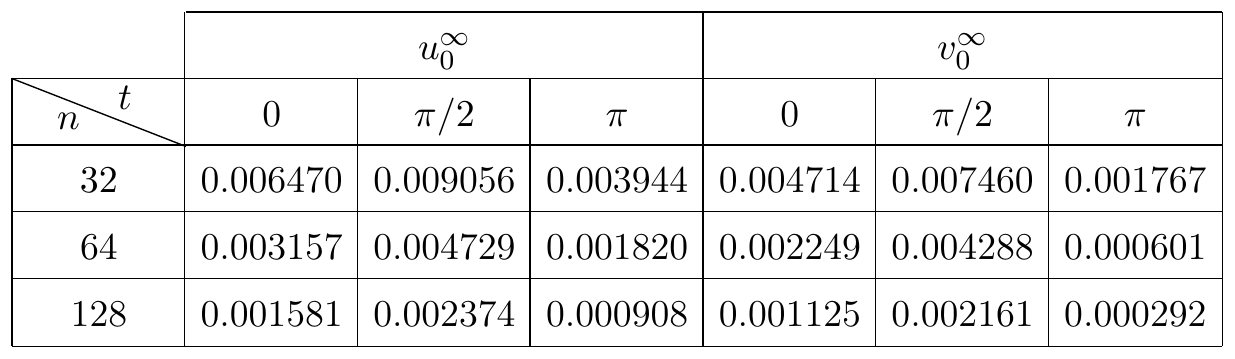}\vspace{0.5cm}
\caption{Absolute errors of the far field patterns of $u_0$ and $v_0$ for different orders $n$ at discrete points $t$.}
\label{table1}
\end{center}
\end{table}

\subsection{Example with oblique incidence}

In this example, we consider the usual obliquely incident (TM) polarized electromagnetic plane wave, resulting to the forms \eqref{incident_el}. We keep the same values for all the parameters as in the previous example. We restrict the computations of the fields to the rectangular domain $[-5,\,5]^2$ and we consider a two-dimension uniform-space discretization, namely, $\bm x_{kj} = (-5 + k \delta , -5 + j\delta),$ where $\delta = 10/(2m-1),$ for $k,j = 0,...,2m-1 .$ We use $m=128 .$

The values of the norms of the scattered electric and magnetic fields $|u_0|, \, |v_0|$ and the interior electric and magnetic fields $|u_1|, \, |v_1|$ are presented in Figures \ref{Fig4} and \ref{Fig5} for different values of the polar angle $\phi ,$ which in $\R^2$ corresponds to the incident direction $(\cos \phi , \,\sin \phi) \in \SQ.$ 

\begin{figure}[p]
\begin{center}
\includegraphics[scale=0.55]{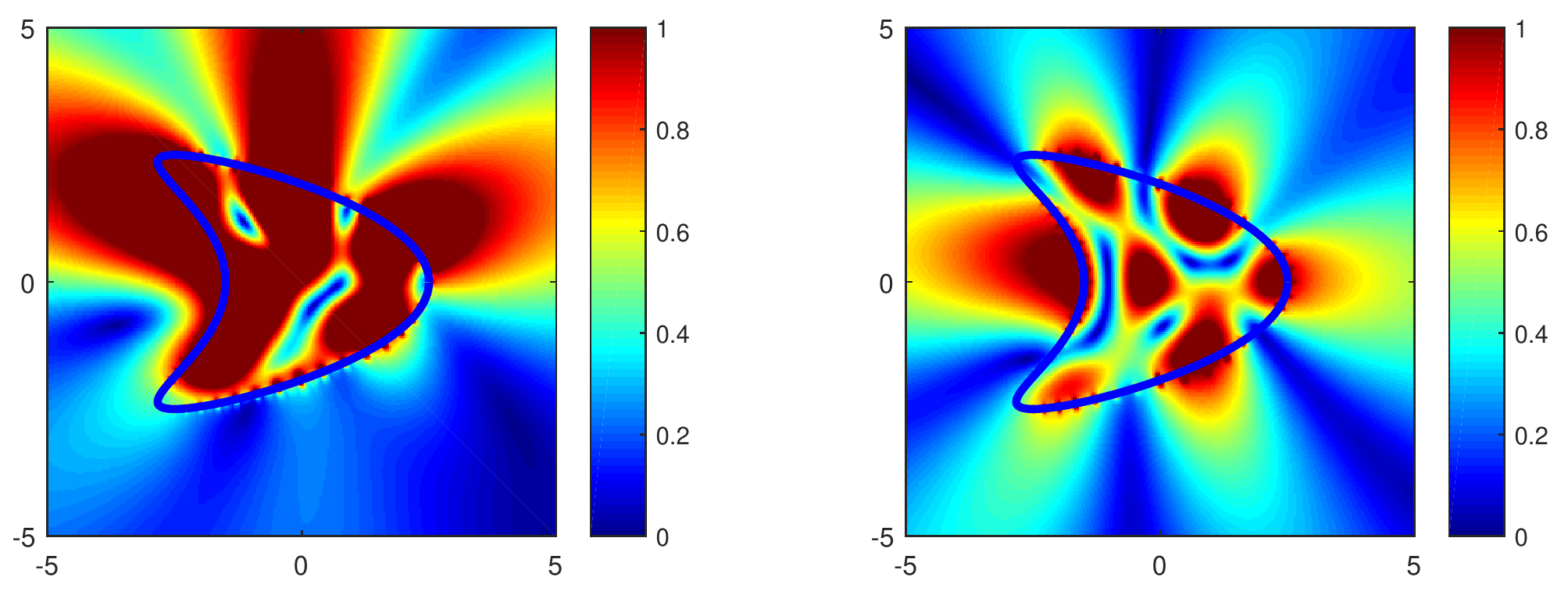}
\caption{The norms of the electric fields $u_0$ and $u_1$ (left) and of the magnetic fields $v_0$ and $v_1$ (right) for $\phi = \pi/2$. }\label{Fig4}
\end{center}
\end{figure}

\begin{figure}[p]
\begin{center}
\includegraphics[scale=0.55]{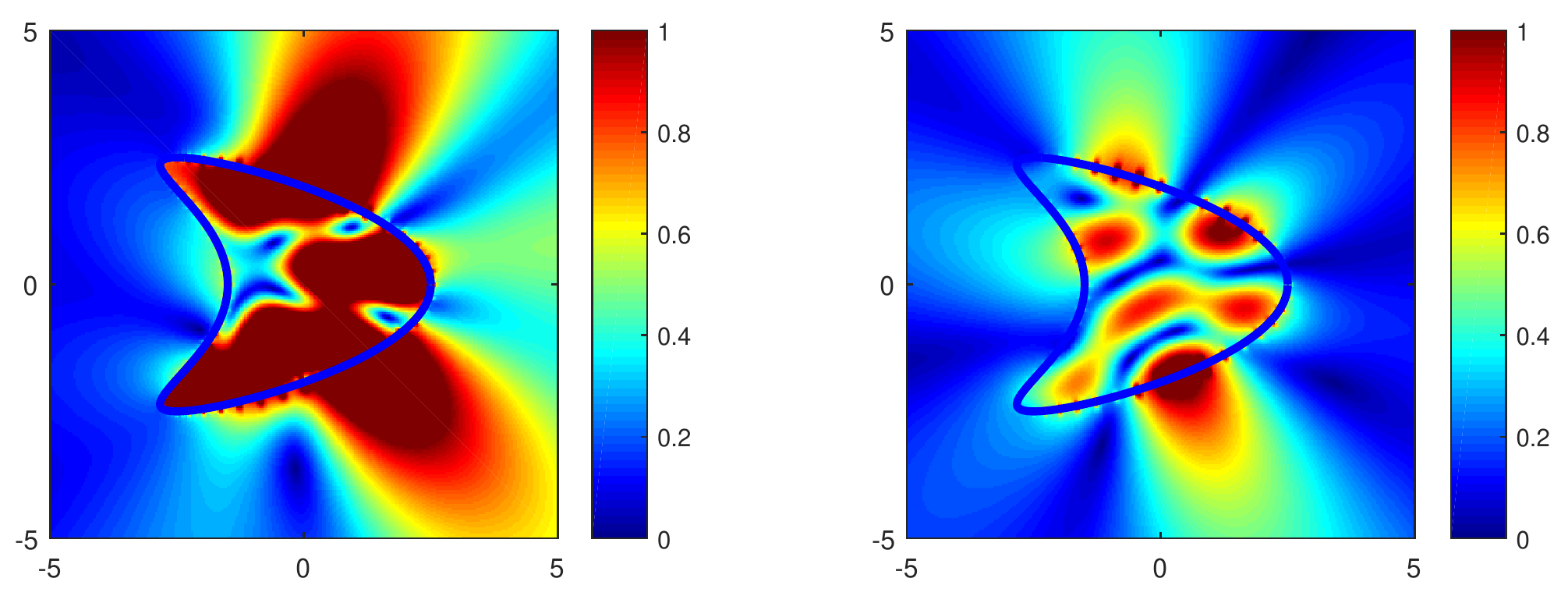}
\caption{The norms of the electric fields $u_0$ and $u_1$ (left) and of the magnetic fields $v_0$ and $v_1$ (right) for $\phi = \pi/9$. }\label{Fig5}
\end{center}
\end{figure}


\section*{Acknowledgements} 

The authors thank the referees for their valuable comments. This research was initiated while DG was visiting the Department of Mathematical Sciences, University of Delaware and he expresses his gratitude for its hospitality. His research was supported by the program of NTUA for sabbatical visits. 
     
\bibliographystyle{amsplain}

\begin{thebibliography}{10}

\bibitem{CakCol06}
F.~Cakoni and D.~Colton, Qualitative methods in inverse scattering
  theory, Springer-Verlag, Berlin, 2006.

\bibitem{CanLee91}
A.C. Cangellaris and R.~Lee, Finite element analysis of electromagnetic
  scattering from inhomogeneous cylinders at oblique incidence, IEEE Trans.
  Ant. Prop. \textbf{39} (1991), 645--650.

\bibitem{ColKre83}
D.~Colton and R.~Kress, Integral equation methods in scattering theory,
  Pure and Applied Mathematics (New York), John Wiley \& Sons Inc., New York,
  1983.

\bibitem{ColKre98}
D.~Colton and R.~Kress, Inverse acoustic and electromagnetic scattering theory, 2 ed.,
  Applied Mathematical Sciences, vol.~93, Springer-Verlag, Berlin, 1998.

\bibitem{CosSte85}
M.~Costabel and E.~Stephan, A direct boundary integral equation method
  for transmission problems, J. Math. Anal. Appl. \textbf{106} (1985),
  367--413.

\bibitem{CouHil62}
R.~Courant and D.~Hilbert, Methods of mathematical physics, vol.~2,
  Wiley-Interscience, New York, 1962.

\bibitem{Isa06}
V.~Isakov, Inverse problems for partial differential equations, Applied
  Mathematical Sciences, vol. 127, Springer, New York, 2006.

\bibitem{KitKle75}
R.~Kittappa and R.E. Kleinman, Acoustic scattering by penetrable
  homogeneous objects, J. Math. Phys. \textbf{16} (1975), 421--432.

\bibitem{KleMar88}
R.E. Kleinman and P.A. Martin, On single integral equations for the
  transmission problem of acoustics, SIAM J. Appl. Math. \textbf{48} (1988),
  no.~2, 307--325.

\bibitem{Kre95}
R.~Kress, On the numerical solution of a hypersingular integral equation
  in scattering theory, J. Comput. Appl. Math. \textbf{61} (1995), no.~3,
  345--360.

\bibitem{Kre99}
R.~Kress, Linear integral equations, 2 ed., Springer Verlag, Berlin,
  1999.

\bibitem{Kre14}
R.~Kress, A collocation method for a hypersingular boundary integral
  equation via trigonometric differentiation, J. Integral Equations Appl.
  \textbf{26} (2014), no.~2, 197--213.

\bibitem{LucPanSche10}
M.~Lucido, G.~Panariello, and F.~Schettiho, Scattering by polygonal
  cross-section dielectric cylinders at oblique incidence, IEEE Trans. Ant.
  Prop. \textbf{58} (2010), 540--551.

\bibitem{Mau49}
A.W. Maue, {\"U}ber die formulierung eines allgemeinen beugungsproblems
  durch eine integralgleichung, Zeitschrift f{\"u}r Physik \textbf{126}
  (1949), 601--618.

\bibitem{Mit66}
K.M. Mitzner, Acoustic scattering from an interface between media of
  greatly different density, J. Math. Phys. \textbf{7} (1966), 2053--2060.

\bibitem{Mon03}
P.~Monk, Finite element methods for maxwell’s equations, Oxford
  University Press, Oxford, 2003.

\bibitem{NakWan13}
G.~Nakamura and H.~Wang, The direct electromagnetic scattering problem
  from an imperfectly conducting cylinder at oblique incidence, J. Math. Anal.
  Appl. \textbf{397} (2013), 142--155.

\bibitem{Ned01}
J.C. N{\'e}d{\'e}lec, Acoustic and electromagnetic equations,
  Springer-Verlag, New York, 2001.

\bibitem{Roj88}
R.G. Rojas, Scattering by an inhomogeneous dielectric/ferrite cylinder of
  arbitrary cross-section shape-oblique incidence case, IEEE Trans. Ant. Prop.
  \textbf{36} (1988), 238--246.

\bibitem{Tsa07}
J.L. Tsalamengas, Exponentially converging nystr{\"o}m methods applied to
  the integral-integrodifferential equations of oblique scattering/ hybrid wave
  propagation in presence of composite dielectric cylinders of arbitrary cross
  section, IEEE Trans. Ant. Prop. \textbf{55} (2007), 3239--3250.

\bibitem{TsiAliAnaKak07}
N.L. Tsitsas, E.G. Alivizatos, H.T. Anastassiu, and D.I. Kaklamani,
  Optimization of the method of auxiliary sources (mas) for oblique
  incidence scattering by an infinite dielectric cylinder, Electrical
  Engineering \textbf{89} (2007), 353--361.

\bibitem{Wai55}
J.R. Wait, Scattering of a plane wave from a circular dielectric cylinder
  at oblique incidence, Canadian J. Phys. \textbf{33} (1955), 189--195.

\bibitem{WanNak12}
H.~Wang and G.~Nakamura, The integral equation method for electromagnetic
  scattering problem at oblique incidence, Appl. Num. Math. \textbf{62}
  (2012), 860--873.

\bibitem{YanGorKis95}
J.~Yan, R.K. Gordon, and A.A. Kishk, Electromagnetic scattering from
  impedance elliptic cylinders using finite difference method, Electromagn.
  \textbf{15} (1995), 157--173.

\bibitem{YouEls97}
H.A. Yousif and A.Z. Elsherbeni, Oblique incidence scattering from two
  eccentric cylinders, J. Electromagnetic Waves Appl. \textbf{11} (1997),
  1273--1288.

\end{thebibliography}

\end{document}